\documentclass[a4paper, 11pt]{article}
\usepackage[a4paper,hmargin={2.5cm,2.5cm},vmargin={3cm,3cm}]{geometry}

\usepackage{amsmath,amssymb}
\usepackage{amsthm}
\usepackage{color}
\usepackage{comment}

\newcommand{\sgn}{\mathrm{sgn}}
\newcommand{\const}{\mathrm{const}}

\newcommand{\Tr}{\mathrm{Tr}}
\newcommand{\Arg}{\mathrm{Arg}}

\renewcommand{\Im}{\mathrm{Im}}
\renewcommand{\Re}{\mathrm{Re}}

\newtheorem{lemma}{Lemma}
\newtheorem{theorem}{Theorem}
\newtheorem{statement}{Statement}

\numberwithin{equation}{section}
\numberwithin{lemma}{section}
\numberwithin{theorem}{section}
\numberwithin{statement}{section}
\numberwithin{definition}{section}

\begin{document}


\begin{center}
\large \textbf{Absence of solitons with sufficient algebraic localization for the Novikov-Veselov equation at nonzero energy}
\end{center}

\begin{center}
A.V. Kazeykina \footnote{Centre des Math\'ematiques Appliqu\'ees, Ecole Polytechnique, Palaiseau, 91128, France; \\ email: kazeykina@cmap.polytechnique.fr}
\end{center}

\textbf{Abstract.} We show that the Novikov--Veselov equation (an analog of KdV in dimension $ 2 + 1 $) at positive and negative energies does not have solitons with the space localization stronger than $ O( | x |^{ -3 } ) $ as $ | x | \to \infty $.

\section{Introduction}
In this paper we are concerned with the Novikov-Veselov equation
\begin{subequations}
\label{NV}
\begin{align}
\label{NV_eq}
& \partial_t v = 4 \Re ( 4 \partial_z^3 v + \partial_z( v w ) - E \partial_{ z } w ), \\
\label{w_def}
& \partial_{ \bar z } w = - 3 \partial_z v, \quad v = \bar v, \quad E \in \mathbb{R},\\
\label{misc}
& v = v( x, t ), \quad w = w( x, t ), \quad x = ( x_1, x_2 ) \in \mathbb{R}^2, \quad t \in \mathbb{R},
\end{align}
\end{subequations}
where the following notations are used
\begin{equation}
\label{derivatives}
\partial_t = \frac{ \partial }{ \partial t }, \quad \partial_z = \frac{ 1 }{ 2 } \left( \frac{ \partial }{ \partial x_1 } - i \frac{ \partial }{ \partial x_2 } \right), \quad \partial_{ \bar z } = \frac{ 1 }{ 2 } \left( \frac{ \partial }{ \partial x_1 } + i \frac{ \partial }{ \partial x_2 } \right).
\end{equation}

Equation (\ref{NV}) is mathematically the most natural $ ( 2 + 1 ) $-dimensional analog of the classic Korteweg-de Vries equation. When $ v = v( x_1, t ) $, $ w = w( x_1, t ) $, equation (\ref{NV}) reduces to KdV. Besides, equation (\ref{NV}) is integrable via the scattering transform for the $ 2 $--dimensional Schr\"odinger equation
\begin{equation}
\label{t_schrodinger}
\begin{aligned}
L \psi = E \psi, & \quad  E = E_{fixed}, \\
L = - \Delta + v( x, t ), & \quad \Delta = 4 \partial_z \partial_{ \bar z }, \quad x \in \mathbb{R}^2.
\end{aligned}
\end{equation}
Note also that tending $ E \to \pm \infty $ in (\ref{NV}) yields another renowned $ ( 2 + 1 ) $-dimensional analog of KdV, Kadomtsev-Petviashvili equation (KP-I and KP-II, respectively).

Equation (\ref{NV}) is contained implicitly in \cite{M} as an equation possessing the following representation
\begin{equation}
\label{L-A-B}
\frac{ \partial( L - E ) }{ \partial t } = [ L - E, A ] + B( L - E ),
\end{equation}
where $ L $ is the operator of the corresponding scattering problem, $ A $, $ B $ are some appropriate differential operators and $ [ \cdot, \cdot ] $ denotes the commutator. For the particular case of the $ 2 $-dimensional Schr\"odinger operator as in (\ref{t_schrodinger}) the following explicit form of $ A $ and $ B $
\begin{equation}
\label{A-B}
\begin{array}{l}
A = - 8 \partial_{ z }^3 - 2 w \partial_{ z } - 8 \partial_{ \bar z }^3 - 2 \bar w \partial_{ \bar z }, \\
B = 2 \partial_{ \bar z } w + 2 \partial_{ \bar z } \bar w,
\end{array} \text{ where } w \text{ is defined via (\ref{w_def}) },
\end{equation}
and the corresponding evolution equation (\ref{NV}) were given in \cite{NV1}, \cite{NV2}, where equation (\ref{NV}) was also studied in the periodic setting.

Solitons and the large time asymptotic behavior of sufficiently localized in space solutions for the Novikov-Veselov equation were studied in the series of works \cite{GN1, G1, Nov2, K1, KN1, KN2, KN3}. In \cite{KN1, K1} it was shown that in the regular case, i.e. when the scattering data are nonsingular at fixed nonzero energy (and for the reflectionless case at positive energy), then these solutions do not contain isolated solitons in the large time asymptotics. In the general case it was shown in \cite{Nov2}, \cite{KN3} that the Novikov-Veselov equation at nonzero energy does not admit exponentially localized solitons. A family of algebraically localized solitons for the Novikov-Veselov equation at positive energy was constructed in \cite{G1} (see also discussion in \cite{KN2}). These solitons are rational functions decaying as $ O\left( | x |^{ - 2 } \right) $ when $ | x | \to \infty $.

Note that KP-I equation possesses soliton solutions and these solutions decay as $ O\left( | x |^{ - 2 } \right) $ when $ | x | \to \infty $. By contrast, KP-II does not possess localized soliton solutions. For the results on existence and nonexistence of localized soliton solutions of KP-I, KP-II and their generalized versions see \cite{BS1}; the symmetry properties and the decay rates of these solutions were derived in \cite{BS2}. For more results on integrable $ ( 2 + 1 ) $-dimensional systems admitting localized soliton solutions, see \cite{AC}, \cite{BLMP2}, \cite{FA}, \cite{FS} and references therein.

In this paper we are concerned with regular, sufficiently localized solutions of (\ref{NV}) satisfying the following conditions
\begin{align}
\label{smoothness}
& \bullet v, w \in C( \mathbb{R}^2 \times \mathbb{R} ), \; v( \cdot, t ) \in C^4( \mathbb{R}^2 ) \quad \forall t \in \mathbb{R}; \\
\label{decrease}
& \bullet | \partial_{ x }^{ j } v( x, t ) | \leqslant \frac{ q( t ) }{ ( 1 + | x | )^{ 3 + j + \varepsilon } }, \; j = ( j_1, j_2 ) \in ( \mathbb{N} \cup 0 )^2, \; j_1 + j_2 \leqslant 4, \text{ for some } q( t ) > 0, \varepsilon > 0; \\
\label{w_decrease}
& \bullet | w( x, t ) | \to 0, \text{ when } | x | \to \infty, \quad t \in \mathbb{R}.
\end{align}
We say that a solution of (\ref{NV}) is a soliton if $ v( x, t ) = V( x - ct ) $ for some $ c = ( c_1, c_2 ) \in \mathbb{R}^2 $. The main result of this paper consists in the following theorem.

\begin{theorem}
\label{main_theorem}
Let $ (v, w) $ be a soliton solution of (\ref{NV}) with $ E \neq 0 $ satisfying properties (\ref{smoothness})-(\ref{w_decrease}). Then $ v \equiv 0 $, $ w \equiv 0 $.
\end{theorem}

To prove this result we consider, in particular, special eigenfunctions of the $ 2 $-dimensional Schr\"odinger operator going back to \cite{F1}, \cite{BLMP1} and we base our reasoning on the ideas proposed in \cite{Nov2}.

Note that Theorem \ref{main_theorem} for the case of zero energy was proved in \cite{K2} for the potentials of conductivity type.

In Section \ref{schr_section} we recall, in particular, some known notions and results from the direct and inverse scattering theory for the two-dimensional Schr\"odinger equation at nonzero energy (see \cite{Nov1}, \cite{G2} and references therein). In addition, we introduce nonzero energy analogs of some ``scattering data'' going back to \cite{BLMP1}. The main result (namely, Theorem \ref{main_theorem}) is proved in Section \ref{main_section}. Section \ref{lemmas_section} contains the proofs of some preliminary lemmas.

This work was fulfilled in the framework of research carried out under the supervision of R.G. Novikov.

\section{Scattering data and inverse scattering equations}
\label{schr_section}

Consider the Schr\"odinger equation on the plane
\begin{equation}
\label{schrodinger}
\begin{aligned}
L \psi = E \psi, & \quad E = E_{ fixed } \in \mathbb{R} \backslash 0, \\
L = - \Delta + v, \quad \Delta & = 4 \partial_z \partial_{ \bar z }, \quad v = v( x ), \quad x \in \mathbb{R}^2
\end{aligned}
\end{equation}
with a potential $ v $ satisfying the following conditions
\begin{equation}
\label{v_conditions}
\begin{aligned}
& v( x ) = \overline{ v( x ) }, \quad v( x ) \in L^{ \infty }( \mathbb{R}^2 ), \\
& | \partial_{ x_1 }^{ j_1 } \partial_{ x_2 }^{ j_2 } v( x ) | < q ( 1 + | x | )^{ - 3 - \varepsilon } \text{ for some } q > 0, \; \varepsilon > 0, \text{ where } j_1, j_2 \in \mathbb{N} \cup 0, \; j_1 + j_2 \leqslant 3.
\end{aligned}
\end{equation}

For equation (\ref{schrodinger}) with $ E > 0 $ we consider its classical scattering eigenfunctions $ \psi^+( x, k ) $, defined for $ k \in \mathbb{R}^2 $, $ k^2 = E $ and specified by
\begin{equation}
\label{scat_eigenfunction}
\psi^+( x, k ) = e^{ i k x } - i \pi \sqrt{ 2 \pi } e^{ -\frac{ i \pi }{ 4 } } f\left( k, | k | \frac{ x }{ | x | } \right) \frac{ e^{ i | k | | x | } }{ \sqrt{ | k | | x | } } + o\left( \frac{ 1 }{ \sqrt{ | x | } } \right), \quad | x | \to \infty,
\end{equation}
with some a priori unknown function $ f $. The function $ f $ is called the scattering amplitude of the potential $ v $. If $ f( k, l ) \equiv 0 $ for $ k, l \in \mathbb{R}^2 $, $ k^2 = l^2 = E $, then the corresponding potential is called transparent (or reflectionless) at fixed energy $ E > 0 $. In this paper we will only be concerned with transparent potentials since it was shown in \cite{Nov2} that the solitons of the Novikov-Veselov equation at positive energy are transparent potentials (see also Lemma \ref{scat_lemma}).

In addition, for equation (\ref{schrodinger}) with $ E \in \mathbb{R} \backslash 0 $ we consider its Faddeev eigenfunctions $ \psi( x, k ) $, defined for $ k \in \Sigma_E $, where
\begin{align*}
& \Sigma_E = \{ k \in \mathbb{C}^2 \colon k^2 = E, \; \Im k \neq 0 \}, \text{ if } E > 0, \\
& \Sigma_E = \{ k \in \mathbb{C}^2 \colon k^2 = E \}, \text{ if } E < 0, \\
\end{align*}
and specified by
\begin{equation}
\label{faddeev_eigenfunction}
\psi( x, k ) = e^{ i k x }( 1 + o( 1 ) ), \quad | x | \to \infty
\end{equation}
(see\cite{F1}, \cite{Nov1}, \cite{G2}).

Finally, for equation (\ref{schrodinger}) with $ E \in \mathbb{R} \backslash 0 $ we will also consider its eigenfunctions $ \varphi( x, k ) $, defined for $ k \in \Sigma_E $ and specified by
\begin{equation}
\label{blmp_eigenfunction}
\varphi( x, k ) = e^{ i k x }( k_1 x_2 - k_2 x_1 + o( 1 ) ), \quad | x | \to \infty.
\end{equation}
These functions are the analogs of solutions introduced in \cite{BLMP1} for the case of zero energy.

Further it will be convenient to assume without loss of generality that $ E = \pm 1 $ (the general case is reduced to this one by a scaling transform) and to introduce the following new variables:
\begin{equation*}
z = x_1 + i x_2, \quad \lambda = \frac{ k_1 + i k_2 }{ \sqrt{E} }.
\end{equation*}
Note that $ k_1 = \frac{ \sqrt{E} }{ 2 } \left( \lambda + \frac{ 1 }{ \lambda } \right) $, $ k_2 = \frac{ i \sqrt{E} }{ 2 } \left( \frac{ 1 }{ \lambda } - \lambda \right) $.

In the new variables $ z \in \mathbb{C} $, $ \lambda \in \mathbb{C} \backslash 0 $ functions $ \psi $ and $ \varphi $ are solutions of (\ref{schrodinger}) with the following asymptotic behavior
\begin{align}
\label{psi_asympt}
& \psi( z, \lambda ) = e^{ \frac{ i \sqrt{ E } }{ 2 }( \lambda \bar z + z / \lambda ) } \mu( z, \lambda ), \quad \mu( z, \lambda ) = 1 + o( 1 ), \text{ as } | z | \to \infty, \\
\label{phi_asympt}
& \varphi( z, \lambda ) = e^{ \frac{ i \sqrt{ E } }{ 2 }( \lambda \bar z + z / \lambda ) } \nu( z, \lambda ), \quad \nu( z, \lambda ) = \frac{ i \sqrt{ E } }{ 2 } \left(\lambda \bar z - \frac{ 1 }{ \lambda } z \right) + o( 1 ), \text { as } | z | \to \infty.
\end{align}
Functions $ \mu( z, \lambda ) $ and $ \nu( z, \lambda ) $ arising in the above formulas can also be defined as solutions of the following integral equations
\begin{align}
\label{mu_int_equation}
& \mu( z, \lambda ) = 1 + \iint\limits_{ \mathbb{C} } g( z - \zeta, \lambda ) v( \zeta ) \mu( \zeta, \lambda ) d \Re \zeta d \Im \zeta, \\
\label{nu_int_equation}
& \nu( z, \lambda ) = \frac{ i \sqrt{E} }{ 2 } \left( \lambda \bar z - \frac{ 1 }{ \lambda } z \right) + \iint\limits_{ \mathbb{C} } g( z - \zeta, \lambda ) v( \zeta ) \nu( \zeta, \lambda ) d \Re \zeta d \Im \zeta, \text{ where } \\
\label{green}
& g( z, \lambda ) = - \left( \frac{ 1 }{ 2 \pi } \right)^2 \iint\limits_{ \mathbb{C} } \frac{ e^{ \frac{ i }{ 2 } ( p \bar z + \bar p z ) } }{ p \bar p + \sqrt{E} ( \lambda \bar p + p / \lambda ) } d \Re p  d \Im p,
\end{align}
where $ z \in \mathbb{C} $, $ \lambda \in \mathbb{C} \backslash 0 $ and, if $ E > 0 $, then $ | \lambda | \neq 1 $.

In terms of  $ m( z, \lambda ) = ( 1 + | z | )^{ -( 2 + \varepsilon / 2 ) } \mu( z, \lambda ) $ and $ n( z, \lambda ) = ( 1 + | z | )^{ -( 2 + \varepsilon / 2 ) } \nu( z, k ) $ equations (\ref{mu_int_equation}) and (\ref{nu_int_equation}), respectively, take the forms
\begin{multline}
\label{m_equation}
m( z, \lambda ) = ( 1 + | z | )^{ -( 2 + \varepsilon / 2 ) } + \\
+ \iint\limits_{ \mathbb{C} } ( 1 + | z | )^{ -( 2 + \varepsilon / 2 ) } g( z - \zeta, \lambda ) \frac{ v( \zeta ) }{ ( 1 + | \zeta | )^{ -( 2 + \varepsilon / 2 ) } } m( \zeta, \lambda ) d \Re \zeta d \Im \zeta,
\end{multline}
\begin{multline}
\label{n_equation}
n( z, \lambda ) = \frac{ i \sqrt{E} }{ 2 } \left( \lambda \bar z - \frac{ 1 }{ \lambda } z \right) ( 1 + | z | )^{ -( 2 + \varepsilon / 2 ) } + \\
+ \iint\limits_{ \mathbb{C} } ( 1 + | z | )^{ -( 2 + \varepsilon / 2 ) } g( z - \zeta, \lambda ) \frac{ v( \zeta ) }{ ( 1 + | \zeta | )^{ -( 2 + \varepsilon / 2 ) } } n( \zeta, \lambda ) d \Re \zeta d \Im \zeta.
\end{multline}
The integral operator $ A( \lambda ) $ of the integral equations (\ref{m_equation}), (\ref{n_equation}) is a Hilbert-Schmidt operator: more precisely, $ A( \cdot, \cdot, \lambda ) \in L^2( \mathbb{C} \times \mathbb{C} ) $, where $ A( z, \zeta, \lambda ) $ is the Schwartz kernel of the integral operator $ A( \lambda ) $, and $ | \Tr A^2( \lambda ) | < \infty $. Thus, the modified Fredholm determinant $ \Delta( \lambda ) $ for (\ref{m_equation}) and (\ref{n_equation}) can be defined by means of the formula:
\begin{equation}
\label{fred_determinant}
\ln \Delta( \lambda ) = \Tr( \ln( I - A( \lambda ) ) + A( \lambda ) ).
\end{equation}
For the precise sense of this definition see \cite{GK}. Considerations of $ \Delta $ go back to \cite{F2}.

We will also define
\begin{equation*}
\mathcal{E} = \{ \lambda \in \Sigma \colon \Delta( \lambda ) = 0 \},
\end{equation*}
where
\begin{equation*}
\Sigma = \mathbb{C} \backslash ( 0 \cup T ) \; \text{ if } E > 0 \quad \text{ and } \quad \Sigma = \mathbb{C} \backslash 0 \; \text{ if } E < 0, \quad T = \{ \lambda \in \mathbb{C} \colon | \lambda | = 1 \}.
\end{equation*}

In this notation $ \mathcal{E} $ represents the set of $ \lambda $ for which either the existence or the uniqueness of the solution of (\ref{schrodinger}) with asymptotics (\ref{psi_asympt}) (or, similarly, of the solution of (\ref{schrodinger}) with asymptotics (\ref{phi_asympt})) fails.

For $ \lambda \in \mathbb{C} \backslash ( \mathcal{E} \cup 0 ) $ we define the following ``scattering data'' for the potential $ v $:
\begin{align}
\label{a_data}
& a( \lambda ) = \iint\limits_{ \mathbb{C} } v( \zeta ) \mu( \zeta, \lambda ) d \Re \zeta d \Im \zeta, \\
& b( \lambda ) = \iint\limits_{ \mathbb{C} } \exp \left( \frac{ i \sqrt{ E } }{ 2 } \left( 1 + ( \sgn E ) \frac{ 1 }{ \lambda \bar \lambda } \right)( ( \sgn E ) \zeta \bar \lambda + \lambda \bar \zeta ) \right) v( \zeta ) \mu( \zeta, \lambda ) d \Re \zeta d \Im \zeta, \\
& \alpha( \lambda ) = \iint\limits_{ \mathbb{C} }  v( \zeta ) \nu( \zeta, \lambda ) d \Re \zeta d \Im \zeta, \\
\label{beta_data}
& \beta( \lambda ) = \iint\limits_{ \mathbb{C} } \exp \left( \frac{ i \sqrt{ E } }{ 2 } \left( 1 + ( \sgn E ) \frac{ 1 }{ \lambda \bar \lambda } \right)( ( \sgn E ) \zeta \bar \lambda + \lambda \bar \zeta ) \right) v( \zeta ) \nu( \zeta, \lambda ) d \Re \zeta d \Im \zeta.
\end{align}

Functions $ a $, $ b $ are the Faddeev generalized scattering data for the $ 2 $-dimensional Schr\"odinger equation. They also arise in a more precise version of expansion (\ref{faddeev_eigenfunction}). The ``scattering data'' $ \alpha $, $ \beta $ for the case of the Schr\"odinger equation at zero energy were introduced in \cite{BLMP1}.

Now we formulate some properties of the introduced functions that will play a substantial role in the proof of the main result.
\begin{statement}[see \cite{F2,HN,Nov1,KN3}]
\label{delta_prop}
Let $ v $ satisfy conditions (\ref{v_conditions}). Then function $ \Delta( \lambda ) $ satisfies the following properties:
\begin{enumerate}
\item \label{delta_continuity} $ \Delta \in C( \bar D_+ ) $, $ \Delta \in C( \bar D_- ) $, where $ \bar D_+ = D_+ \cup \partial D_+ $, $ D_+ = \{ \lambda \in \mathbb{C} \colon | \lambda | < 1 \} $, $ \bar D_- = D_- \cup \partial D_- $, $ D_- = \{ \lambda \in \mathbb{C} \colon | \lambda | > 1 \} $;

if $  E < 0 $ or if $ E > 0 $ and $ v $ is transparent, then $ \Delta \in C( \mathbb{C} ) $;
\item \label{delta_limit} $ \Delta( \lambda ) \to 1 $ as $ | \lambda | \to \infty $, $ | \lambda | \to 0 $;
\item \label{real_valued} $ \Delta $ is real-valued;
\item $ \Delta( \lambda ) $ satisfies the following $ \bar \partial $-equation
\begin{equation}
\label{dif_det}
\frac{ \partial \Delta }{ \partial \bar \lambda } = - \frac{ \sgn( \lambda \bar \lambda - 1 ) }{ 4 \pi \bar \lambda } \left( a\left( - ( \sgn E ) \frac{ 1 }{ \bar \lambda } \right) - \hat v( 0 ) \right) \Delta,
\end{equation}
where $ \hat v( 0 ) = \iint\limits_{ \mathbb{C} } v( \zeta ) d \Re \zeta d \Im \zeta $, $ \lambda \in \mathbb{C} \backslash ( T \cup \mathcal{E} \cup 0 ) $, $ T = \{ \lambda \in \mathbb{C} \colon | \lambda | = 1 \} $;
\item \label{delta_sym} $ \Delta( \lambda ) = \Delta\left( - (\sgn E) \frac{ 1 }{ \bar \lambda } \right) $, $ \lambda \in \mathbb{C} \backslash 0 $;
\item \label{constT} if $ E < 0 $ or if $ E > 0 $ and $ v $ is transparent, then $ \Delta \equiv \const $ on $ T = \{ \lambda \in \mathbb{C} \colon | \lambda | = 1 \} $.
\end{enumerate}
\end{statement}
Note that $ \Delta \not \in C( \mathbb{C} ) $ for $ E > 0 $, in general. In this case $ \Delta $ on $ \bar D_{ \pm } $ is considered as an extension from $ D_{ \pm } $.

If $ v $ satisfies assumptions (\ref{v_conditions}), then functions $ a( \lambda ) $, $ b( \lambda ) $, $ \alpha( \lambda ) $, $ \beta( \lambda ) $ are continuous on $ \mathbb{C} \backslash ( \mathcal{E} \cup 0 ) $. Note also (see \cite{HN}) that
\begin{equation}
\label{fourrier_v}
a( \lambda ) \to \hat v( 0 ) \text{ as } \lambda \to 0, \lambda \to \infty.
\end{equation}
If the scattering data $ a $ are well-defined on the unit circle $ T $, then its behavior on $ T $ is described as follows:
\begin{equation}
\label{a_behavior}
\begin{aligned}
& a \equiv 0 \text{ on } T \text{ if } E > 0 \text{ and } v \text{ is transparent (see \cite{GN2})}, \\
& a \equiv b \text{ on } T \text{ if } E < 0.
\end{aligned}
\end{equation}

If $ v $ satisfies assumptions (\ref{v_conditions}) and, in the case of positive energy, $ v $ is transparent, i.e. $ f \equiv 0 $ at fixed energy, then the function $ \mu( z, \lambda ) $, defined by (\ref{mu_int_equation}), satisfies the following properties (see \cite{GN3}, \cite{Nov2}, \cite{G2} and references therein):
\begin{equation}
\label{mu_continuity}
\mu( z, \lambda ) \text{ is a continuous function of } \lambda \text{ on } \mathbb{C} \backslash ( 0 \cup \mathcal{E} );
\end{equation}

\begin{subequations}
\label{mu_props}
\begin{align}
\label{mu_dbar}
& \frac{ \partial \mu( z, \lambda ) }{ \partial \bar \lambda } = r( z, \lambda ) \overline{ \mu( z, \lambda ) }, \\
\label{r}
& r( z, \lambda ) = r( \lambda ) \exp\left( -\frac{ i \sqrt{ E } }{ 2 } \left( 1 + ( \sgn E ) \frac{ 1 }{ \lambda \bar \lambda } \right) \left( (\sgn E) \bar \lambda z + \lambda \bar z \right) \right), \\
\label{t}
& r( \lambda ) = \frac{ \sgn( \lambda \bar \lambda - 1 ) }{ 4 \pi \bar \lambda } b( \lambda )
\end{align}
\end{subequations}
for $ \lambda \in \mathbb{C} \backslash ( 0 \cup \mathcal{E} \cup T ) $, where $ T = \{ \lambda \in \mathbb{C} \colon | \lambda | = 1 \} $;

\begin{equation}
\label{mu_limits}
\mu \to 1, \text{ as } \lambda \to \infty, \; \lambda \to 0.
\end{equation}

Inverse scattering equations (\ref{mu_props}) together with conditions (\ref{mu_continuity}) and (\ref{mu_limits}) determine uniquely function $ \mu $ from nonsingular scattering data $ b $, i.e. when $ \mathcal{E} = \varnothing $. Potential $ v $ (transparent for the case of positive energy) can then be found from the following formula
\begin{equation}
\label{v_repres}
v( z ) = 2 i \sqrt{E} \frac{ \partial \mu_{ -1 }( z ) }{ \partial z },
\end{equation}
where $ \mu_{ -1 }( z ) $ is defined via the following expansion
\begin{equation}
\label{mu_lambda_asympt}
\mu( z, \lambda ) = 1 + \frac{ \mu_{ -1 }( z ) }{ \lambda } + o\left( \frac{ 1 }{ | \lambda | } \right), \text{ as } \lambda \to \infty.
\end{equation}

\section{Proof of Theorem \ref{main_theorem}}
\label{main_section}
We will start this section by formulating some preliminary lemmas. The proofs of Lemmas \ref{shift_lemma}, \ref{dyn_lemma} are given in Section \ref{lemmas_section}.

\begin{lemma}
\label{w_lemma}
Let $ v \in C^4( \mathbb{C} ) $ and
\begin{equation}
\label{assum}
| \partial^{ j_1 }_{ z } \partial^{ j_2 }_{ \bar z } v( z ) | < \frac{ q }{ ( 1 + | z | )^{ j + 3 + \varepsilon } }, \quad j_1, j_2 \in \{ 0 \cup \mathbb{N} \}, \, j = j_1 + j_2 \leqslant 4
\end{equation}
for $ \forall z \in \mathbb{C} $ and some $ q > 0 $, $ \varepsilon > 0 $.
Let $ w $ be defined by
\begin{align*}
& \partial_{ \bar z } w = - 3 \partial_{ z } v, \\
& w( z ) \to 0 \text{ as } z \to \infty.
\end{align*}
Then
\begin{align}
\label{w_asympt}
& w( z ) = \frac{ 3 \hat v( 0 ) }{ \pi z^2 } + O\left( \frac{ 1 }{ | z |^3 } \right), \text{ as } z \to \infty, \\
\label{part_w_asympt}
& \partial_{ z } w( z ) = - \frac{ 6 \hat v( 0 ) }{ \pi z^3 } + O\left( \frac{ 1 }{ | z |^4 } \right), \text{ as } z \to \infty,
\end{align}
where $ \hat v( 0 ) = \iint\limits_{ \mathbb{C} } v( \zeta ) d \Re \zeta d \Im \zeta $.
\end{lemma}
The proof of representations (\ref{w_asympt}), (\ref{part_w_asympt}) is easily derived from the Taylor formula for $ W( \xi ) = w\left( \frac{ 1 }{ \xi } \right) $ and $ - \xi^2 W( \xi ) = \partial_{ z } w( z ) |_{ z = \frac{ 1 }{ \xi } } $, correspondingly, in the neighborhood of $ \xi = 0 $.

Let us denote by
\begin{equation}
\label{s_data}
S( \lambda ) = \{ a( \lambda ), b( \lambda ), \alpha( \lambda ), \beta( \lambda ) \}, \quad \lambda \in \mathbb{C} \backslash ( \mathcal{E} \cup 0 ),
\end{equation}
the scattering data for a potential $ v $, defined by (\ref{a_data})-(\ref{beta_data}) in the framework of equation (\ref{schrodinger}).

\begin{lemma}
\label{shift_lemma}
Let $ v( z ) $ be a potential satisfying (\ref{v_conditions}) with the scattering data $  S( \lambda ) $, $ \lambda \in \mathbb{C} \backslash ( \mathcal{E} \cup 0 ) $. Then the scattering data $ S_{ \eta }( \lambda ) $ for the potential $ v_{ \eta }( z ) = v( z - \eta ) $ are defined for $ \lambda \in \mathbb{C} \backslash ( \mathcal{E} \cup 0 ) $ and are related to $ S( \lambda ) $ by the following formulas
\begin{align}
\label{a_eta}
& a_{ \eta }( \lambda ) = a( \lambda ), \\
\label{b_eta}
& b_{ \eta }( \lambda ) =  \exp\left\{ \frac{ i \sqrt{ E } }{ 2 } \left( 1 + ( \sgn E ) \frac{ 1 }{ \lambda \bar \lambda } \right)\left( ( \sgn E ) \bar \lambda \eta + \lambda \bar \eta \right) \right\} b( \lambda ), \\
\label{alpha_eta}
& \alpha_{ \eta }( \lambda ) = \alpha( \lambda ) + \frac{ i \sqrt{ E } }{ 2 } \left( \lambda \bar \eta - \frac{ 1 }{ \lambda } \eta \right) a( \lambda ), \\
\label{beta_eta}
& \beta_{ \eta }( \lambda ) = \exp\left\{ \frac{ i \sqrt{ E } }{ 2 } \left( 1 + ( \sgn E ) \frac{ 1 }{ \lambda \bar \lambda } \right)\left( ( \sgn E ) \bar \lambda \eta + \lambda \bar \eta \right) \right\} \left( \beta( \lambda ) + \frac{ i \sqrt{ E } }{ 2 } \left( \lambda \bar \eta - \frac{ 1 }{ \lambda } \eta \right) b( \lambda ) \right).
\end{align}
\end{lemma}

\begin{lemma}
\label{dyn_lemma}
Let $ ( v, w ) $ satisfy equation (\ref{NV}) and conditions (\ref{smoothness})-(\ref{w_decrease}). Let $ S( \lambda, t ) $ be the scattering data for $ v $ defined by (\ref{s_data}) for a certain $ \lambda \in \mathbb{C} \backslash ( \mathcal{E} \cup 0 ) $ and all $ t \in \mathbb{R} $. Then the evolution of these scattering data is described as follows:
\begin{align}\
\label{a_evolution}
& a( \lambda, t ) = a( \lambda, 0 ), \\
\label{b_evolution}
& b( \lambda, t ) = \exp \left\{ i ( \sqrt{E} )^3 \left( \lambda^3 + \frac{ 1 }{ \lambda^3 } + ( \sgn E ) \left( \bar \lambda^3 + \frac{ 1 }{ \bar \lambda^3 } \right) \right) t \right\} b( \lambda, 0 ), \\
\label{alpha_evolution}
& \alpha( \lambda, t ) = \alpha( \lambda, 0 ) + 3 i ( \sqrt{E} )^3 \left( \lambda^3 - \frac{ 1 }{ \lambda^3 } \right) ( a( \lambda, 0 ) - \hat v( 0 ) ) t, \\
\label{beta_evolution}
& \beta( \lambda, t ) = \exp \left\{ i ( \sqrt{E} )^3 \left( \lambda^3 + \frac{ 1 }{ \lambda^3 } + ( \sgn E ) \left( \bar \lambda^3 + \frac{ 1 }{ \bar \lambda^3 } \right) \right) t \right\} \left( \beta( \lambda, 0 ) + 3 i ( \sqrt{E} )^3 \left( \lambda^3 - \frac{ 1 }{ \lambda^3 } \right) b( \lambda, 0 ) t \right).
\end{align}
\end{lemma}

\begin{lemma}[see \cite{Nov2}]
\label{scat_lemma}
Let $ ( v, w ) $ satisfy equation (\ref{NV}) for some $ E > 0 $ and conditions (\ref{smoothness})-(\ref{w_decrease}). In addition, let $ v $ be a soliton, i.e. $ v( x, t ) = V( x - c t ) $ for some $ c = ( c_1, c_2 ) \in \mathbb{R}^2 $. Then $ f( k, l ) \equiv 0 $, $ k, l \in \mathbb{R}^2 $, $ k^2 = l^2 = E > 0 $, where $ f $ is the scattering amplitude for the potential $ v $ in the framework of the Schr\"odinger equation (\ref{t_schrodinger}).
\end{lemma}

The concluding part of the proof of Theorem \ref{main_theorem} consists in the following. First of all, if $ ( v, w ) $ is a soliton solution of equation (\ref{NV}) for some $ E > 0 $, then $ v $ is transparent due to Lemma \ref{scat_lemma}.

Further, since $ ( v, w ) $ is a soliton, from Lemma \ref{shift_lemma} it follows that the set $ \mathcal{E} $ of values of $ \lambda \in \mathbb{C} $ for which the scattering data $ a( \lambda ) $, $ b( \lambda ) $, $ \alpha( \lambda ) $, $ \beta( \lambda ) $ are not well-defined does not depend on $ t $.

Since $ ( v, w ) $ is a soliton, the time dynamics of its scattering data $ b $ is described by the formula
\begin{equation*}
b( \lambda, t ) = \exp\left\{ \frac{ i \sqrt{E} }{ 2 } \left( \left( \lambda + ( \sgn E ) \frac{ 1 }{ \bar \lambda } \right) \bar c t + \left( ( \sgn E ) \bar \lambda + \frac{ 1 }{ \lambda } \right) c t \right) \right\} b( \lambda, 0 ),
\end{equation*}
where notation $ c = c_1 + i c_2 $ is used (see formula (\ref{b_eta}) of Lemma \ref{shift_lemma}). Combining this with (\ref{b_evolution}) from Lemma \ref{dyn_lemma} gives
\begin{multline*}
\exp \left\{ i ( \sqrt{E} )^3 \left( \lambda^3 + \frac{ 1 }{ \lambda^3 } + ( \sgn E ) \left( \bar \lambda^3 + \frac{ 1 }{ \bar \lambda^3 } \right) \right) t \right\} b( \lambda, 0 ) = \\
= \exp\left\{ \frac{ i \sqrt{E} }{ 2 } \left( \left( \lambda + ( \sgn E ) \frac{ 1 }{ \bar \lambda } \right) \bar c t + \left( ( \sgn E ) \bar \lambda + \frac{ 1 }{ \lambda } \right) c t \right) \right\} b( \lambda, 0 ).
\end{multline*}
Since functions $ \lambda $, $ \bar \lambda $, $ \lambda^3 $, $ \bar \lambda^3 $, $ \frac{ 1 }{ \lambda } $, $ \frac{ 1 }{ \bar \lambda } $, $ \frac{ 1 }{ \lambda^3 } $, $ \frac{ 1 }{ \bar \lambda^3 } $, $ 1 $ are linearly independent in any open nonempty neighborhood of any point in $ \mathbb{C} \backslash 0 $ and $ b( \lambda, 0 ) $ is continuous on $ \mathbb{C} \backslash ( \mathcal{E} \cup 0 ) $, we obtain that $ b( \lambda, 0 ) \equiv 0 $ on $ \mathbb{C} \backslash ( \mathcal{E} \cup 0 ) $.

Similarly, from (\ref{alpha_eta}), (\ref{alpha_evolution}) we get that
\begin{equation*}
\alpha( \lambda, 0 ) + \frac{ i \sqrt{E} }{ 2 } \left( \lambda \bar c - \frac{ 1 }{ \lambda } c \right) t a( \lambda, 0 ) = \alpha( \lambda, 0 ) + 3 i ( \sqrt{E} )^3 \left( \lambda^3 - \frac{ 1 }{ \lambda^3 } \right) t ( a( \lambda, 0 ) - \hat v( 0 ) )
\end{equation*}
or
\begin{equation}
\label{key_expression}
a( \lambda, 0 ) = \frac{ 3 i ( \sqrt{E} )^3 \left( \lambda^3 - \frac{ 1 }{ \lambda^3 } \right) \hat v( 0 ) }{ 3 i ( \sqrt{E} )^3 \left( \lambda^3 - \frac{ 1 }{ \lambda^3 } \right) - \frac{ i \sqrt{E} }{ 2 } \left( \lambda \bar c - \frac{ 1 }{ \lambda } c \right) }.
\end{equation}

We will prove that expression (\ref{key_expression}) implies that $ \hat v( 0 ) = 0 $ and thus $ a( \lambda, 0 ) \equiv 0 $ for $ \lambda \in \mathbb{C} \backslash ( \mathcal{E} \cup 0 ) $.

Expression (\ref{key_expression}) is well-defined everywhere except for the points which are the roots of equation
\begin{equation}
\label{denominator}
3 i ( \sqrt{E} )^3 \left( \lambda^3 - \frac{ 1 }{ \lambda^3 } \right) - \frac{ i \sqrt{E} }{ 2 } \left( \lambda \bar c - \frac{ 1 }{ \lambda } c \right) = 0.
\end{equation}
Evidently, equation (\ref{denominator}) has six roots $ \lambda_j $, $ j = 1, \ldots, 6 $, counted with multiplicity. Equation (\ref{denominator}) was studied in detail in \cite{KN1} (see Lemma 3.1 of \cite{KN1}). In particular, it was shown that for any value of parameter $ c $
\begin{equation}
\label{roots_on T}
\text{ equation (\ref{denominator}) has at least two roots on $ T = \{ \lambda \in \mathbb{C} \colon | \lambda | = 1 \} $}.
\end{equation}

We will consider two subcases depending on whether $ \Delta $ vanishes on $ T $ or not. Note that item \ref{constT} of Statement \ref{delta_prop} implies that if $ \Delta $ vanishes in one point of $ T $, then it vanishes in every point of $ T $.

\medskip
\noindent
(i) $ \Delta \neq 0 $ on $ T $

\medskip
In this subcase $ a( \lambda, 0 ) $ and $ b( \lambda, 0 ) $ are well-defined on $ T $. Thus $ b( \lambda, 0 ) \equiv 0 $ on T. From (\ref{a_behavior}) we obtain that $ a( \lambda, 0 ) \equiv 0 $ on $ T $. In view of representation (\ref{key_expression}) and property (\ref{roots_on T}) this can only hold if $ \hat v( 0 ) = 0 $.

\medskip
\noindent
(ii) $ \Delta \equiv 0 $ on $ T $

\medskip
Recall that due to item \ref{real_valued} of Statement \ref{delta_prop} $ \Delta \in \mathbb{R} $ and due to item \ref{delta_limit} of Statement \ref{delta_prop} $ \Delta( 0 ) > 0 $. Consider
\begin{equation}
\label{l_definition}
l_{ \varphi } = \{ r e^{ i \varphi }, 0 \leqslant r < r' \leqslant 1 \colon \Delta( r e^{ i \varphi } ) > 0, \, \Delta( r' e^{ i \varphi } ) = 0 \},
\end{equation}
where $ \varphi \in ( -\pi, \pi ] $ is some angle, such that $ \varphi \neq \Arg \lambda_j $, $ j = 1, \ldots, 6 $, where $ \lambda_j $ are the roots of equation (\ref{denominator}). We will also denote by $ \bar l_{ \varphi } $ the closure of $ l_{ \varphi } $.

Denote
\begin{equation}
\label{absn_u}
u( \bar \lambda ) = - \frac{ \sgn( \lambda \bar \lambda - 1 ) }{ 4 \pi \bar \lambda } \left( a\left( -(\sgn E) \frac{ 1 }{ \bar \lambda } \right) - \hat v( 0 ) \right), \quad \lambda \in l_{ \varphi },
\end{equation}
where $ a $ is given by (\ref{key_expression}). Note that from item \ref{delta_sym} of Statement \ref{delta_prop} it follows that the scattering data $ a( - \sgn E / \bar \lambda ) $ are well-defined for $ \lambda \in l_{ \varphi } $. Using item \ref{real_valued} of Statement \ref{delta_prop} we obtain that
\begin{equation*}
\frac{ \partial \ln \Delta }{ \partial \lambda } = \overline{ u( \bar \lambda ) }, \quad \frac{ \partial \ln \Delta }{ \partial \bar \lambda } = u( \bar \lambda ), \quad \lambda \in l_{ \varphi },
\end{equation*}
and
\begin{equation}
\label{direction_derivative}
\frac{ \partial \ln \Delta }{ \partial \gamma } = ( u( \bar \lambda ) + \overline{ u( \bar \lambda ) } ) \cos \varphi + i ( u( \bar \lambda ) - \overline{ u( \bar \lambda ) } ) \sin \varphi, \quad \gamma \in S^1, \, \gamma = ( \cos \varphi, \sin \varphi ),
\end{equation}
where $ \frac{ \partial \ln \Delta }{ \partial \gamma } $ denotes the directional derivative of $ \ln \Delta $ along the direction $ \gamma $.
From (\ref{key_expression}) it follows that $ u $ is bounded as $ \lambda \to 0 $. Thus by integrating (\ref{direction_derivative}) along $ l_{ \varphi } $ we obtain that
\begin{equation}
\label{U_def}
\ln \Delta( \lambda ) = U( \bar \lambda ) + \overline{ U( \bar \lambda ) }, \quad  U( \bar \lambda ) = \int\limits_{ 0 }^{ \bar \lambda } u( \bar \zeta ) d \bar \zeta, \quad \lambda \in l_{ \varphi },
\end{equation}
or
\begin{equation}
\label{Delta_def}
\Delta( \lambda ) = \Phi( \bar \lambda ) \overline{ \Phi( \bar \lambda ) }, \text{ where } \Phi( \bar \lambda ) = \exp U( \bar \lambda ), \quad \lambda \in l_{ \varphi }.
\end{equation}

From well-definedness of $ a $, given by expression (\ref{key_expression}), on $ \bar l_{ \varphi } $ and continuity of $ \Delta $ (see item \ref{delta_continuity} of Statement \ref{delta_prop}) it follows that (\ref{Delta_def}) is valid for $ \lambda \in \bar l_{ \varphi } $. In particular, expression (\ref{Delta_def}) implies that
\begin{equation}
\label{tilde_not_vanish}
\forall \lambda \in \bar l_{ \varphi } \quad \Delta( \lambda ) \neq 0.
\end{equation}
However, from definition (\ref{l_definition}) we have that for $ r' e^{ i \varphi } \in \bar l_{ \varphi } $ $ \Delta( r' e^{ i \varphi } ) = 0 $ which contradicts (\ref{tilde_not_vanish}). Thus we have shown that subcase (ii), when $ \Delta \equiv 0 $ on $ T $, cannot hold.

We have demonstrated that $ \hat v( 0 ) = 0 $ and, consequently, $ a( \lambda, 0 ) \equiv 0 $ on $ \lambda \in \mathbb{C} \backslash ( \mathcal{E} \cup 0 ) $.

Equation (\ref{dif_det}) together with the established fact that $ a \equiv 0 $ on $ \mathbb{C} \backslash \mathcal{E} $ implies that $ \Delta $ is holomorphic on $ \mathbb{C} \backslash ( \mathcal{E} \cup T \cup 0 ) $, where $ T = \{ \lambda \in \mathbb{C} \colon | \lambda | = 1 \} $. From properties \ref{delta_continuity}, \ref{delta_limit} of Statement \ref{delta_prop} it follows that $ \Delta $ is holomorphic on $ \mathbb{C} \backslash ( \mathcal{E} \cup T ) $.

Suppose now that $ \mathcal{E} \neq \varnothing $. Since $ \mathcal{E} $ is a closed set, then there exists $ \lambda_* \in \mathcal{E} $ such that $ | \lambda_* | = \min\limits_{ \lambda \in \mathcal{E} } | \lambda | $. Note that property \ref{delta_limit} of Statement \ref{delta_prop} implies that $ | \lambda_* | > 0 $.

If $ | \lambda_* | \geqslant 1 $, then $ \Delta( \lambda ) $ is holomorphic on $ D_+ = \{ \lambda \in \mathbb{C} \colon | \lambda | < 1 \} $ and properties \ref{delta_limit}, \ref{real_valued} of Statement \ref{delta_prop} imply that $ \Delta \equiv 1 $ on $ D_+ $. If $ | \lambda_* | < 1 $, then $ \Delta( \lambda ) $ is holomorphic on the set $ D_+^h = \{ \lambda \in \mathbb{C} \colon | \lambda | < \lambda_* \} $ and properties \ref{delta_limit}, \ref{real_valued} imply that $ \Delta \equiv 1 $ on $ D_+^h $. On the other hand, $ \Delta( \lambda_* ) = 0 $, which contradicts property \ref{delta_continuity} from Statement \ref{delta_prop}. Thus we have proved that $ \Delta( \lambda ) \equiv 1 $ on $ D_+ $. Property \ref{delta_sym} of Statement \ref{delta_prop} implies that $ \Delta( \lambda ) \equiv 1 $ on $ D_- = \{ \lambda \in \mathbb{C} \colon | \lambda | > 1 \} $. Finally, from \ref{delta_continuity} of Statement \ref{delta_prop} it follows that $ \Delta \equiv 1 $ on $ \mathbb{C} $.

The function $ \mu $ is holomorphic on $ \mathbb{C} $ as follows from (\ref{mu_props}), (\ref{mu_limits}) and the established facts that $ \mathcal{E} = \varnothing $, $ b \equiv 0 $. The function $ \mu $ is also bounded due to the property (\ref{mu_limits}). From Liouville's theorem it follows that $ \mu \equiv 1 $. Then, finally, from (\ref{v_repres}), (\ref{mu_lambda_asympt}) we obtain that $ v \equiv 0 $.

\section{Proofs of Lemmas \ref{shift_lemma}, \ref{dyn_lemma}}
\label{lemmas_section}

\begin{proof}[Proof of Lemma \ref{shift_lemma}]
Formulas (\ref{a_eta}), (\ref{b_eta}) were derived in \cite{KN2} for the case of negative energy. Here we present their full derivation for both cases of positive and negative energies.

We note that $ \psi( z - \eta, \lambda ) $ satisfies (\ref{schrodinger}) with $ v_{ \eta }( z ) $ and has the asymptotics
\begin{equation*}
\psi( z - \eta, \lambda ) = e^{ \frac{ i \sqrt{ E } }{ 2 } ( \lambda ( \bar z - \bar \eta ) + ( z - \eta ) / \lambda ) } ( 1 + o( 1 ) ),
\end{equation*}
as $ | z | \to \infty $. Thus for Faddeev eigenfunction $ \psi_{ \eta }( z, \lambda ) $ corresponding to potential $ v_{ \eta }( z ) $ we obtain the following representation: $ \psi_{ \eta }( z, \lambda ) = e^{ \frac{ i \sqrt{E} }{ 2 } ( \lambda \bar \eta + \eta / \lambda ) }\psi( z - \eta, \lambda ) $. Consequently, for function $ \mu_{ \eta }( z, \lambda ) $ corresponding to function $ \psi_{ \eta }( z, \lambda ) $ and defined via (\ref{psi_asympt}) we have $ \mu_{ \eta }( z, \lambda ) = \mu( z - \eta, \lambda ) $.

For the scattering data we have
\begin{equation*}
a_{ \eta }( \lambda ) = \iint\limits_{ \mathbb{C} } v_{ \eta }( \zeta ) \mu_{ \eta }( \zeta, \lambda ) d \Re \zeta d \Im \zeta = \iint\limits_{ \mathbb{C} } v( \zeta - \eta ) \mu( \zeta - \eta, \lambda ) d \Re \zeta d \Im \zeta = a( \lambda )
\end{equation*}
and
\begin{multline}
b_{ \eta }( \lambda ) = \iint\limits_{ \mathbb{C} } \exp\left\{ \frac{ i \sqrt{E} }{ 2 } \left( 1 + ( \sgn E ) \frac{ 1 }{ \lambda \bar \lambda } \right) \left( ( \sgn E ) \bar \lambda \zeta + \lambda \bar \zeta \right) \right\} v_{ \eta }( \zeta ) \mu_{ \eta }( \zeta, \lambda ) d \Re \zeta d \Im \zeta = \\
= \iint\limits_{ \mathbb{C} } \exp\left\{ \frac{ i \sqrt{E} }{ 2 } \left( 1 + ( \sgn E ) \frac{ 1 }{ \lambda \bar \lambda } \right) \left( ( \sgn E ) \bar \lambda \zeta + \lambda \bar \zeta \right) \right\} v( \zeta - \eta ) \mu( \zeta - \eta, \lambda ) d \Re \zeta d \Im \zeta = \\
= \exp\left\{ \frac{ i \sqrt{E} }{ 2 } \left( 1 + ( \sgn E ) \frac{ 1 }{ \lambda \bar \lambda } \right) \left( ( \sgn E ) \bar \lambda \eta + \lambda \bar \eta \right) \right\} b( \lambda ).
\end{multline}

Similarly, to derive formulas (\ref{alpha_eta}), (\ref{beta_eta}), we note that $ \varphi( z - \eta, \lambda ) $ satisfies (\ref{schrodinger}) with $ v_{ \eta }( z ) $ and has the asymptotics
\begin{equation*}
\varphi( z - \eta, \lambda ) = e^{ \frac{ i \sqrt{ E } }{ 2 } ( \lambda ( \bar z - \bar \eta ) + ( z - \eta ) / \lambda ) } \left( \frac{ i \sqrt{ E } }{ 2 } \left( \lambda ( \bar z - \bar \eta ) - \frac{ 1 }{ \lambda }( z - \eta ) \right) + o( 1 ) \right),
\end{equation*}
as $ | z | \to \infty $. Thus for eigenfunction $ \varphi_{ \eta }( z, \lambda ) $ of equation (\ref{schrodinger}) with potential $ v_{ \eta }( z ) $ satisfying asymptotics (\ref{phi_asympt}) we obtain the following representation: $ \varphi_{ \eta }( z, \lambda ) = e^{ \frac{ i \sqrt{E} }{ 2 } ( \lambda \bar \eta + \eta / \lambda ) }( \varphi( z - \eta, \lambda ) + \frac{ i \sqrt{E} }{ 2 }  \left( \lambda \bar \eta - \frac{ 1 }{ \lambda } \eta \right) \psi( z - \eta, \lambda ) ) $. Consequently, for function $ \nu_{ \eta }( z, \lambda ) $ corresponding to function $ \varphi_{ \eta }( z, \lambda ) $ and defined via (\ref{phi_asympt}) we have $ \nu_{ \eta }( z, \lambda ) = \nu( z - \eta, \lambda ) + \frac{ i \sqrt{E} }{ 2 }  \left( \lambda \bar \eta - \frac{ 1 }{ \lambda } \eta \right) \mu( z - \eta, \lambda ) $.

For the scattering data we have
\begin{multline*}
\alpha_{ \eta }( \lambda ) = \iint\limits_{ \mathbb{C} } v_{ \eta }( \zeta ) \nu_{ \eta }( \zeta, \lambda ) d \Re \zeta d \Im \zeta = \\
= \iint\limits_{ \mathbb{C} } v( \zeta - \eta ) \nu( \zeta - \eta, \lambda ) d \Re \zeta d \Im \zeta + \frac{ i \sqrt{E} }{ 2 }  \left( \lambda \bar \eta - \frac{ 1 }{ \lambda } \eta \right) \iint\limits_{ \mathbb{C} } v( \zeta - \eta ) \mu( \zeta - \eta, \lambda ) d \Re \zeta d \Im \zeta = \\
= \alpha( \lambda ) + \frac{ i \sqrt{E} }{ 2 } \left( \lambda \bar \eta - \frac{ 1 }{ \lambda } \eta \right) a( \lambda )
\end{multline*}
and
\begin{multline}
\beta_{ \eta }( \lambda ) = \iint\limits_{ \mathbb{C} } \exp\left\{ \frac{ i \sqrt{E} }{ 2 } \left( 1 + ( \sgn E ) \frac{ 1 }{ \lambda \bar \lambda } \right) \left( ( \sgn E ) \bar \lambda \zeta + \lambda \bar \zeta \right) \right\} v_{ \eta }( \zeta ) \nu_{ \eta }( \zeta, \lambda ) d \Re \zeta d \Im \zeta = \\
= \iint\limits_{ \mathbb{C} } \exp\left\{ \frac{ i \sqrt{E} }{ 2 } \left( 1 + ( \sgn E ) \frac{ 1 }{ \lambda \bar \lambda } \right) \left( ( \sgn E ) \bar \lambda \zeta + \lambda \bar \zeta \right) \right\} v( \zeta - \eta ) \nu( \zeta - \eta, \lambda ) d \Re \zeta d \Im \zeta + \\
+ \frac{ i \sqrt{E} }{ 2 }  \left( \lambda \bar \eta - \frac{ 1 }{ \lambda } \eta \right) \iint\limits_{ \mathbb{C} } \exp\left\{ \frac{ i \sqrt{E} }{ 2 } \left( 1 + ( \sgn E ) \frac{ 1 }{ \lambda \bar \lambda } \right) \left( ( \sgn E ) \bar \lambda \zeta + \lambda \bar \zeta \right) \right\} v( \zeta - \eta ) \mu( \zeta - \eta, \lambda ) d \Re \zeta d \Im \zeta = \\
= \exp\left\{ \frac{ i \sqrt{E} }{ 2 } \left( 1 + ( \sgn E ) \frac{ 1 }{ \lambda \bar \lambda } \right) \left( ( \sgn E ) \bar \lambda \eta + \lambda \bar \eta \right) \right\} \left( \beta( \lambda ) + \frac{ i \sqrt{E} }{ 2 }  \left( \lambda \bar \eta - \frac{ 1 }{ \lambda } \eta \right) b( \lambda ) \right).
\end{multline}

\end{proof}

In order to prove Lemma \ref{dyn_lemma} we introduce the following operator
\begin{equation}
\label{t_operator}
T = \partial_{ t } - 8 \partial_{ z }^3 - 2 w \partial_{ z } - 8 \partial_{ \bar z }^3 - 2 \bar w \partial_{ \bar z },
\end{equation}
where $ w $ is defined via (\ref{w_def}), (\ref{w_decrease}) for some potential $ v $. Note that $ T = \partial_{ t } + A $, where $ A $ is the operator of  (\ref{A-B}).

We will need the following auxiliary lemma describing how $ T $ acts on the spectral solutions of the two-dimensional Schr\"odinger equation.
\begin{lemma}
\label{aux_lemma}
Let $ ( v, w ) $ satisfy conditions (\ref{smoothness})-(\ref{w_decrease}) and
\begin{equation*}
| \partial_{ t } v( x, t ) | \leqslant \frac{ \tilde q( t ) }{ ( 1 + | x | )^{ 3 + \varepsilon } }, \text{ for some } \tilde q( t ) > 0.
\end{equation*}

Suppose that for a certain $ \lambda \in \mathbb{C} $, $ | \lambda | \neq 1 $, and $ t $ belonging to a certain interval $ t \in ( t_1, t_2 ) $ the solution $ \psi( z, \lambda, t ) $ of (\ref{t_schrodinger}) with asymptotics (\ref{psi_asympt}) exists and is unique. Similarly, suppose that the solution $ \varphi( z, \lambda, t ) $ of (\ref{t_schrodinger}) with asymptotics (\ref{phi_asympt}) exists and is unique. Then
\begin{align}
\label{tpsi_asympt}
& T \psi =  i ( \sqrt{E} )^3 e^{ \frac{ i \sqrt{E} }{ 2 } ( \lambda \bar z + z / \lambda ) } \left( \left( \lambda^3 + \frac{ 1 }{ \lambda^3 } \right) + o( 1 ) \right), \text{ as } | z | \to \infty, \\
\label{tphi_asympt}
& T \varphi = i ( \sqrt{E} )^3 e^{ \frac{ i \sqrt{E} }{ 2 } ( \lambda \bar z + z / \lambda ) } \left( \frac{ i \sqrt{E} }{ 2 } \left( \lambda^3 + \frac{ 1 }{ \lambda^3 } \right) \left( \lambda \bar z - \frac{ 1 }{ \lambda } z \right) + 3 \left( \lambda^3 - \frac{ 1 }{ \lambda^3 } \right) + o( 1 ) \right), \quad \text{ as } | z | \to \infty.
\end{align}
\end{lemma}
\begin{proof}
First of all, due to Lemma \ref{w_lemma}, in order to demonstrate (\ref{tpsi_asympt}), (\ref{tphi_asympt}) it is sufficient to show that
\begin{align}
\label{mu_limit_prop}
& \partial_{ t } \mu \to 0, \quad \partial_{ z }^j \mu \to 0, \quad \partial_{ \bar z }^j \mu \to 0, \quad j = 1, 2, 3, \text{ as } | z | \to \infty, \\
\label{nu_limit_prop}
& \partial_{ t } \nu \to 0, \quad \partial_{z} \nu \to -\frac{ i \sqrt{E} }{ 2 \lambda }, \quad \partial_{ \bar z } \nu \to \frac{ i \sqrt{E} }{ 2 } \lambda, \quad \partial_{ z }^k \nu \to 0, \quad \partial_{ \bar z }^k \nu \to 0, \quad k = 2, 3, \text{ as } | z | \to \infty,
\end{align}
where $ \mu( z, \lambda, t ) = e^{ -\frac{ i \sqrt{E} }{ 2 }( \lambda \bar z + z / \lambda ) } \psi( z, \lambda, t ) $, $ \nu( z, \lambda, t ) = e^{ -\frac{ i \sqrt{E} }{ 2 }( \lambda \bar z + z / \lambda ) } \varphi( z, \lambda, t ) $. We will only prove properties (\ref{mu_limit_prop}). Properties (\ref{nu_limit_prop}) are proved similarly.

Function $ \mu $ is defined as the solution of the integral equation (\ref{mu_int_equation}), where the notation (\ref{green}) is used. Differentiating (\ref{mu_int_equation}) with respect to $ t $ yields the following integral equation for $ \partial_{ t } \mu $:
\begin{equation}
\label{tmu_equation}
\partial_{ t } \mu( z, \lambda, t ) = \iint\limits_{ \mathbb{C} } g( z - \zeta, \lambda ) \partial_{ t } v( \zeta, t ) \mu( \zeta, \lambda, t ) d \Re \zeta d \Im \zeta+ \iint\limits_{ \mathbb{C} } g( z - \zeta, \lambda ) v( \zeta, t ) \partial_{ t } \mu( \zeta, \lambda, t ) d \Re \zeta d \Im \zeta.
\end{equation}

Differentiating (\ref{mu_int_equation}) $ j $ times with respect to $ z $ yields the following integral equation for $ \partial^j_{ z } \mu $:
\begin{equation*}
\partial^j_{ z } \mu( z, \lambda, t ) = \iint\limits_{ \mathbb{C} } \partial^j_{ z } g( z - \zeta, \lambda ) v( \zeta, t ) \mu( \zeta, \lambda, t ) d \Re \zeta d \Im \zeta, \quad j = 1, 2, 3.
\end{equation*}
We integrate this equation by parts, taking into account property (\ref{decrease}) of function $ v $ and the fact that $ \partial^j_{ z } g( z - \zeta, \lambda ) = ( -1 )^j \partial^{ j }_{ \zeta } g( z - \zeta, \lambda ) $. Thus we obtain
\begin{equation}
\label{zmu_equation}
\partial^j_{ z } \mu( z, \lambda, t ) = \iint\limits_{ \mathbb{C} } g( z - \zeta, \lambda ) \partial^j_{ \zeta } ( v( \zeta, t ) \mu( \zeta, \lambda, t ) ) d \Re \zeta d \Im \zeta, \quad j = 1, 2, 3.
\end{equation}
Similarly, $ \partial^j_{ \bar z } \mu $ satisfies the following integral equation
\begin{equation}
\label{barzmu_equation}
\partial^j_{ \bar z } \mu( z, \lambda, t ) = \iint\limits_{ \mathbb{C} } g( z - \zeta, \lambda ) \partial^j_{ \bar \zeta } ( v( \zeta, t ) \mu( \zeta, \lambda, t ) ) d \Re \zeta d \Im \zeta, \quad j = 1, 2, 3.
\end{equation}

Equation (\ref{zmu_equation}) is an equation on the unknown function $ \partial^j_{ z } \mu $, where it is assumed that functions $ \partial^k_{ z } \mu $, $ k < j $, are already defined. Similarly equation (\ref{barzmu_equation}) is an equation on the unknown function $ \partial^j_{ \bar z } \mu $, where it is assumed that functions $ \partial^k_{ \bar z } \mu $, $ k < j $, are already defined. The assumptions of lemma imply that for each of the equations (\ref{tmu_equation}), (\ref{zmu_equation}), (\ref{barzmu_equation}) its solution exists, is unique and can be represented as $ ( 1 + | z | )^{ 2 + \varepsilon/2 } u( z, \lambda, t ) $ with some corresponding $ u( \cdot, \lambda, t ) \in L^2( \mathbb{C} ) $.

It was shown in \cite{Nov1} that the function $ g $ defined by (\ref{green}) possesses the following property: $ | g( z, \lambda ) | \leqslant \frac{ \const }{ | z | } $ for $ \forall \lambda \in \mathbb{C} $, $ | \lambda | \neq 1 $ and sufficiently large $ z $, and $ | g( z, \lambda ) | \leqslant \const \ln | z | $ for $ \forall \lambda \in \mathbb{C} $ and sufficiently small $ z $. This property implies that
\begin{equation}
\label{g_property}
\begin{aligned}
& \left| \iint\limits_{ \mathbb{C} } g( z - \zeta, \lambda ) U( \zeta ) d \Re \zeta d \Im \zeta \right| \to 0 \text{ as } | z | \to \infty \\
& \text{ for any } U \in L^1( \mathbb{C} )\cap L^2( \mathbb{C} ).
\end{aligned}
\end{equation}

Let us denote by $ \xi( z, \lambda, t ) $ the solution of any of the equations (\ref{mu_int_equation}), (\ref{tmu_equation})-(\ref{barzmu_equation}). As noted before, $ \xi( z, \lambda, t ) $ can be represented in the form $ \xi( z, \lambda, t ) = ( 1 + | z | )^{ 2 + \varepsilon / 2 } u( z, \lambda, t ) $ for some $ u( \cdot, \lambda, t ) \in L^2( \mathbb{C} ) $. Then from assumptions on $ v $ it follows that $ \partial^{ j }_{ z } v( \cdot, t ) \xi( \cdot, \lambda, t ) \in L^1( \mathbb{C} ) \cap \L^2( \mathbb{C} ) $, $ \partial^{ j }_{ \bar z } v( \cdot, t ) \xi( \cdot, \lambda, t ) \in L^1( \mathbb{C} ) \cap L^2( \mathbb{C} ) $ for $ j = 0, \dots, 3 $ and $ \partial_t v( \cdot, t ) \xi( \cdot, \lambda, t ) \in L^1( \mathbb{C} ) \cap L^2( \mathbb{C} ) $. Thus, from (\ref{g_property}) with $ U( \cdot ) = \partial_t v( \cdot, t ) \mu( \cdot, \lambda, t ) $ and $ U( \cdot ) = v( \cdot, t ) \partial_t \mu( \cdot, \lambda, t ) $ it follows that the right part of (\ref{tmu_equation}) tends to zero as $ | z | \to \infty $. Similarly, considering equations (\ref{zmu_equation}), (\ref{barzmu_equation}) consecutively we obtain that the right part of each of these equations tends to zero as $ | z | \to \infty $. Consequently, $ \partial_{ t } \mu \to 0 $, $ \partial_{ z }^j \mu \to 0 $, $ \partial_{ \bar z }^j \mu \to 0 $, $ j = 1, 2, 3 $, as $ | z | \to \infty $.

\end{proof}

\begin{proof}[Proof of Lemma \ref{dyn_lemma}]
Formulas (\ref{a_evolution}), (\ref{b_evolution}) have already been known in literature (see for example \cite{G2}). Since their derivation is similar to the derivation of formulas (\ref{alpha_evolution}), (\ref{beta_evolution}), we confine ourselves to the derivation of the latter. The derivation of analogs of formulas (\ref{a_evolution})-(\ref{beta_evolution}) for the case of zero energy can be found in \cite{BLMP1}.

Equation (\ref{NV}) represents a condition under which the following is true
\begin{equation}
\label{equiv_formulation}
[ T, L ] \eta = E T \eta, \quad \forall \eta \colon L \eta = E \eta,
\end{equation}
where $ L $ is defined in (\ref{schrodinger}) and $ T $ is defined in (\ref{t_operator}) (see \cite{M}, \cite{BLMP1}).

Let us take $ \eta = \varphi $, where $ \varphi $ is the solution of $ L \varphi = E \varphi $ with the asymptotics (\ref{phi_asympt}). Then (\ref{equiv_formulation}) implies
\begin{equation*}
L T \varphi = E \varphi.
\end{equation*}
From Lemma \ref{aux_lemma} we have that
\begin{equation*}
T \varphi = i ( \sqrt{E} )^3 e^{ \frac{ i \sqrt{E} }{ 2 } ( \lambda \bar z + z / \lambda ) } \left( \frac{ i \sqrt{E} }{ 2 } \left( \lambda^3 + \frac{ 1 }{ \lambda^3 } \right) \left( \lambda \bar z - \frac{ 1 }{ \lambda } z \right) + 3 \left( \lambda^3 - \frac{ 1 }{ \lambda^3 } \right) + o( 1 ) \right), \quad | z | \to \infty.
\end{equation*}
The uniqueness of the solution of (\ref{schrodinger}) with the asymptotics (\ref{psi_asympt}) at the considered value of $ \lambda $ implies that
\begin{equation*}
T \varphi = i ( \sqrt{E} )^3 \left( \lambda^3 + \frac{ 1 }{ \lambda^3 } \right) \varphi + 3 i ( \sqrt{E} )^3 \left( \lambda^3 - \frac{ 1 }{ \lambda^3 } \right) \psi.
\end{equation*}
In other words,
\begin{equation}
\label{partial_phi}
\partial_{ t } \varphi = 8 \partial_{z}^3 \varphi + 2 w \partial_{ z } \varphi + 8 \partial_{ \bar z }^3 \varphi + 2 \bar w \partial_{ \bar z } \varphi + i ( \sqrt{E} )^3 \left( \lambda^3 + \frac{ 1 }{ \lambda^3 } \right) \varphi + 3 i ( \sqrt{E} )^3 \left( \lambda^3 - \frac{ 1 }{ \lambda^3 } \right) \psi.
\end{equation}
Now we write $ \alpha( \lambda, t ) $ in the form
\begin{equation*}
\alpha( \lambda, t ) = \iint\limits_{ \mathbb{C} } e^{ -\frac{ i \sqrt{E} }{ 2 } ( \lambda \bar \zeta + \zeta / \lambda ) } v( \zeta, t ) \varphi( \zeta, \lambda, t ) d \Re \zeta d \Im \zeta
\end{equation*}
and compute its derivative with respect to time:
\begin{equation}
\label{partial_alpha}
\partial_{ t } \alpha( \lambda, t ) = \iint\limits_{ \mathbb{C} } e^{ -\frac{ i \sqrt{E} }{ 2 } ( \lambda \bar \zeta + \zeta / \lambda ) } \partial_{ t } v( \zeta, t ) \varphi( \zeta, \lambda, t ) d \Re \zeta d \Im \zeta + \iint\limits_{ \mathbb{C} } e^{ -\frac{ i \sqrt{E} }{ 2 } ( \lambda \bar \zeta + \zeta / \lambda ) } v( \zeta, t ) \partial_{ t } \varphi( \zeta, \lambda, t ) d \Re \zeta d \Im \zeta.
\end{equation}
Substituting (\ref{NV}) and (\ref{partial_phi}) into (\ref{partial_alpha}), integrating the resulting expression by parts and taking into account that $ - 4 \partial_{ \zeta } \partial_{ \bar \zeta } \varphi + v \varphi = E \varphi $, we obtain
\begin{equation}
\label{partial_concise_alpha}
\partial_{ t } \alpha( \lambda, t ) = 3 i ( \sqrt{E} )^3 \left( \lambda^3 - \frac{ 1 }{ \lambda^3 } \right) ( a( \lambda, t ) - \hat v( 0 ) )
\end{equation}
(see Appendix for the detailed derivation of this formula). Formulas (\ref{a_evolution}), (\ref{partial_concise_alpha}) yield (\ref{alpha_evolution}).

Similarly, we write $ \beta( \lambda, t ) $ in the form
\begin{equation*}
\beta( \lambda, t ) = \iint\limits_{ \mathbb{C} } e^{ \frac{ \sqrt{-E} }{ 2 } ( \bar \lambda \zeta + \bar \zeta / \bar \lambda ) } v( \zeta, t ) \varphi( \zeta, \lambda, t ) d \Re \zeta d \Im \zeta
\end{equation*}
and compute its derivative with respect to time:
\begin{equation}
\label{partial_beta}
\partial_{ t } \beta( \lambda, t ) = \iint\limits_{ \mathbb{C} } e^{ \frac{ \sqrt{-E} }{ 2 } ( \bar \lambda \zeta + \bar \zeta / \bar \lambda ) } \partial_{ t } v( \zeta, t ) \varphi( \zeta, \lambda, t ) d \Re \zeta d \Im \zeta + \iint\limits_{ \mathbb{C} } e^{ \frac{ \sqrt{-E} }{ 2 } ( \bar \lambda \zeta + \bar \zeta / \bar \lambda ) } v( \zeta, t ) \partial_{ t } \varphi( \zeta, \lambda, t ) d \Re \zeta d \Im \zeta.
\end{equation}
Substituting (\ref{NV}) and (\ref{partial_phi}) into (\ref{partial_beta}), integrating the resulting expression by parts and taking into account that $ - 4 \partial_{ \zeta } \partial_{ \bar \zeta } \varphi + v \varphi = E \varphi $, we obtain
\begin{equation}
\label{partial_concise_beta}
\partial_{ t } \beta( \lambda, t ) = i ( \sqrt{E} )^3 \left( \lambda^3 + \frac{ 1 }{ \lambda^3 } + ( \sgn E ) \left( \bar \lambda^3 + \frac{ 1 }{ \bar \lambda^3 } \right) \right) \beta( \lambda, t ) + 3 i ( \sqrt{E} )^3 \left( \lambda^3 - \frac{ 1 }{ \lambda^3 } \right) b( \lambda, t )
\end{equation}
(see Appendix for the detailed derivation of this formula). Using formula (\ref{b_evolution}), we obtain (\ref{beta_evolution}).
\end{proof}

\appendix
\section{Appendix}
Here we present the detailed derivation of formulas (\ref{partial_concise_alpha}), (\ref{partial_concise_beta}) proceeding from representations (\ref{partial_alpha}), (\ref{partial_beta}), respectively.

\bigskip
\bigskip
\noindent
\emph{Derivation of (\ref{partial_concise_alpha}).} Substituting (\ref{NV}) and (\ref{partial_phi}) into (\ref{partial_alpha}) yields
\begin{multline}
\label{sum_fourteen_alpha}
\partial_{ t } \alpha( \lambda, t ) = \lim\limits_{ R \to \infty } \Biggl( 8 \iint\limits_{ B_R } e^{ -\frac{ i \sqrt{E} }{ 2 } ( \lambda \bar \zeta + \frac{ \zeta }{ \lambda } ) } \, \partial_{ \zeta }^3 v \, \varphi \, d \Re \zeta \, d \Im \zeta + 8 \iint\limits_{ B_R } e^{ -\frac{ i \sqrt{E} }{ 2 } ( \lambda \bar \zeta + \frac{ \zeta }{ \lambda } ) } \, \partial_{ \bar \zeta }^3 v \, \varphi \, d \Re \zeta \, d \Im \zeta + \\
+ 2 \iint\limits_{ B_R } e^{ -\frac{ i \sqrt{E} }{ 2 } ( \lambda \bar \zeta + \frac{ \zeta }{ \lambda } ) } \, \partial_{ \zeta } v \, w \, \varphi \, d \Re \zeta \, d \Im \zeta + 2 \iint\limits_{ B_R } e^{ -\frac{ i \sqrt{E} }{ 2 } ( \lambda \bar \zeta + \frac{ \zeta }{ \lambda } ) } \, v \, \partial_{ \zeta } w \, \varphi \, d \Re \zeta \, d \Im \zeta + \\
+ 2 \iint\limits_{ B_R } e^{ -\frac{ i \sqrt{E} }{ 2 } ( \lambda \bar \zeta + \frac{ \zeta }{ \lambda } ) } \, \partial_{ \bar \zeta } v \, \bar w \, \varphi \, d \Re \zeta \, d \Im \zeta + 2 \iint\limits_{ B_R } e^{ -\frac{ i \sqrt{E} }{ 2 } ( \lambda \bar \zeta + \frac{ \zeta }{ \lambda } ) } \, v \, \partial_{ \bar \zeta } \bar w \, \varphi \, d \Re \zeta \, d \Im \zeta - \\
- 2 E \iint\limits_{ B_R } e^{ -\frac{ i \sqrt{E} }{ 2 } ( \lambda \bar \zeta + \frac{ \zeta }{ \lambda } ) } \, \partial_{ \zeta } w \, \varphi \, d \Re \zeta \, d \Im \zeta - 2 E \iint\limits_{ B_R } e^{ -\frac{ i \sqrt{E} }{ 2 } ( \lambda \bar \zeta + \frac{ \zeta }{ \lambda } ) } \, \partial_{ \bar \zeta } \bar w \, \varphi \, d \Re \zeta \, d \Im \zeta + \\
+ 8 \iint\limits_{ B_R } e^{ -\frac{ i \sqrt{E} }{ 2 } ( \lambda \bar \zeta + \frac{ \zeta }{ \lambda } ) } \, v \, \partial_{ \zeta }^{ 3 } \varphi \, d \Re \zeta \, d \Im \zeta + 8 \iint\limits_{ B_R } e^{ -\frac{ i \sqrt{E} }{ 2 } ( \lambda \bar \zeta + \frac{ \zeta }{ \lambda } ) } \, v \, \partial_{ \bar \zeta }^{ 3 } \varphi \, d \Re \zeta \, d \Im \zeta + \\
+ 2 \iint\limits_{ B_R } e^{ -\frac{ i \sqrt{E} }{ 2 } ( \lambda \bar \zeta + \frac{ \zeta }{ \lambda } ) } \, v \, w \, \partial_{ \zeta } \varphi \, d \Re \zeta \, d \Im \zeta + 2 \iint\limits_{ B_R } e^{ -\frac{ i \sqrt{E} }{ 2 } ( \lambda \bar \zeta + \frac{ \zeta }{ \lambda } ) } \, v \, \bar w \, \partial_{ \bar \zeta } \varphi \, d \Re \zeta \, d \Im \zeta + \\
+ i ( \sqrt{E} )^3 \left( \lambda^3 + \frac{ 1 }{ \lambda^3 } \right) \iint\limits_{ B_R } e^{ -\frac{ i \sqrt{E} }{ 2 } ( \lambda \bar \zeta + \frac{ \zeta }{ \lambda } ) } \, v \, \varphi \, d \Re \zeta \, d \Im \zeta + \\
+ 3 i ( \sqrt{E} )^3 \left( \lambda^3 - \frac{ 1 }{ \lambda^3 } \right) \iint\limits_{ B_R } e^{ -\frac{ i \sqrt{E} }{ 2 } ( \lambda \bar \zeta + \frac{ \zeta }{ \lambda } ) } \, v \, \psi \, d \Re \zeta \, d \Im \zeta \Biggr) = \lim\limits_{ R \to \infty } \sum\limits_{ i = 1 }^{ 14 } I_i.
\end{multline}
Integrating $ I_9 $ by parts yields
\begin{multline*}
\lim\limits_{ R \to \infty } I_9 = - \frac{ i ( \sqrt{E} )^3 }{ \lambda^3 } \iint\limits_{ \mathbb{C} } e^{ -\frac{ i \sqrt{E} }{ 2 } ( \lambda \bar \zeta + \frac{ \zeta }{ \lambda } ) } \, v \, \varphi \, d \Re \zeta \, d \Im \zeta + \frac{ 6 E }{ \lambda^2 } \iint\limits_{ \mathbb{C} } e^{ -\frac{ i \sqrt{E} }{ 2 } ( \lambda \bar \zeta + \frac{ \zeta }{ \lambda } ) } \, \partial_{ \zeta } v \, \varphi \, d \Re \zeta \, d \Im \zeta + \\
 +\frac{ 12 i \sqrt{E} }{ \lambda } \iint\limits_{ \mathbb{C} } e^{ -\frac{ i \sqrt{E} }{ 2 } ( \lambda \bar \zeta + \frac{ \zeta }{ \lambda } ) } \, \partial_{ \zeta }^2 v \, \varphi \, d \Re \zeta \, d \Im \zeta - 8 \iint\limits_{ \mathbb{C} } e^{ -\frac{ i \sqrt{E} }{ 2 } ( \lambda \bar \zeta + \frac{ \zeta }{ \lambda } ) } \, \partial_{ \zeta }^3 v \, \varphi \, d \Re \zeta \, d \Im \zeta.
\end{multline*}
In this way it can be obtained that
\begin{multline}
\label{first_sum_alpha}
\lim\limits_{ R \to \infty } ( I_1 + I_2 + I_9 + I_{ 10 } + I_{ 13 } ) =\\
= \frac{ 6 E }{ \lambda^2 } \iint\limits_{ \mathbb{C} } e^{ -\frac{ i \sqrt{E} }{ 2 } ( \lambda \bar \zeta + \frac{ \zeta }{ \lambda } ) } \, \partial_{ \zeta } v \, \varphi \, d \Re \zeta \, d \Im \zeta + \frac{ 12 i \sqrt{E} }{ \lambda } \iint\limits_{ \mathbb{C} } e^{ \frac{ - i \sqrt{E} }{ 2 } ( \lambda \bar \zeta + \frac{ \zeta }{ \lambda } ) } \, \partial_{ \zeta }^2 v \, \varphi \, d \Re \zeta \, d \Im \zeta + \\
 +6 E \lambda^2 \iint\limits_{ \mathbb{C} } e^{ -\frac{ i \sqrt{E} }{ 2 } ( \lambda \bar \zeta + \frac{ \zeta }{ \lambda } ) } \, \partial_{ \bar \zeta } v \, \varphi \, d \Re \zeta \, d \Im \zeta + 12 i \sqrt{E} \lambda \iint\limits_{ \mathbb{C} } e^{ -\frac{ i \sqrt{E} }{ 2 } ( \lambda \bar \zeta + \frac{ \zeta }{ \lambda } ) } \, \partial_{ \bar \zeta }^2 v \, \varphi \, d \Re \zeta \, d \Im \zeta.
\end{multline}

Integrating $ I_{ 11 } $ by parts we obtain
\begin{multline*}
\lim\limits_{ R \to \infty } ( I_{ 11 } + I_{ 7 } ) = - 2 \iint\limits_{ \mathbb{C} } e^{ -\frac{ i \sqrt{E} }{ 2 } ( \lambda \bar \zeta + \frac{ \zeta }{ \lambda } ) } \, \partial_{ \zeta } v \, w \, \varphi \, d \Re \zeta \, d \Im \zeta + \\
+ \lim\limits_{ R \to \infty } \Biggl( \frac{ i \sqrt{E} }{ \lambda } \iint\limits_{ B_R } e^{ -\frac{ i \sqrt{E} }{ 2 } ( \lambda \bar \zeta + \frac{ \zeta }{ \lambda } ) } \, v \, w \, \varphi \, d \Re \zeta \, d \Im \zeta - 2 E \iint\limits_{ B_R } e^{ -\frac{ i \sqrt{E} }{ 2 } ( \lambda \bar \zeta + \frac{ \zeta }{ \lambda } ) } \, \partial_{ \zeta } w \, \varphi \, d \Re \zeta \, d \Im \zeta \Biggr) - \\
- 2 \iint\limits_{ \mathbb{C} } e^{ -\frac{ i \sqrt{E} }{ 2 } ( \lambda \bar \zeta + \frac{ \zeta }{ \lambda } ) } \, v \, \partial_{ \zeta } w \, \varphi \, d \Re \zeta \, d \Im \zeta.
\end{multline*}

Consider
\begin{equation*}
\tilde I = \frac{ i \sqrt{E} }{ \lambda } \iint\limits_{ B_R } e^{ -\frac{ i \sqrt{E} }{ 2 } ( \lambda \bar \zeta + \frac{ \zeta }{ \lambda } ) } \, v \, w \, \varphi \, d \Re \zeta \, d \Im \zeta.
\end{equation*}
Taking into account that $ - 4 \partial_{ \zeta } \partial_{ \bar \zeta } \varphi + v \varphi = E \varphi $ we obtain
that
\begin{equation*}
\tilde I = \frac{ i ( \sqrt{E} )^3 }{ \lambda } \iint\limits_{ B_R } e^{ -\frac{ i \sqrt{E} }{ 2 } ( \lambda \bar \zeta + \frac{ \zeta }{ \lambda } ) } \, w \, \varphi \, d \Re \zeta \, d \Im \zeta + \frac{ 4 i \sqrt{E} }{ \lambda } \iint\limits_{ B_R } e^{ -\frac{ i \sqrt{E} }{ 2 } ( \lambda \bar \zeta + \frac{ \zeta }{ \lambda } ) } \, w \, \partial_{ \zeta } \partial_{ \bar \zeta } \varphi \, d \Re \zeta \, d \Im \zeta = \tilde I_{ (1) } + \tilde I_{ (2) }.
\end{equation*}

Applying the Stokes theorem to $ \tilde I_{ (2) } $ we get that
\begin{equation}
\label{app_i_two}
\tilde I_{ (2) } = - \frac{ 2 \sqrt{ E } }{ \lambda } \int\limits_{ C_R }  e^{ -\frac{ i \sqrt{E} }{ 2 } ( \lambda \bar \zeta + \frac{ \zeta }{ \lambda } ) } \, w \, \partial_{ \bar \zeta } \varphi \, d \bar \zeta  - \frac{ 4 i \sqrt{E} }{ \lambda } \iint\limits_{ B_R } \partial_{ \zeta } \left( e^{ -\frac{ i \sqrt{E} }{ 2 } ( \lambda \bar \zeta + \frac{ \zeta }{ \lambda } ) } \, w \right) \, \partial_{ \bar \zeta } \varphi \, d \Re \zeta \, d \Im \zeta,
\end{equation}
where $ C_R $ is the boundary of $ B_R $.

Note that
\begin{equation}
\label{app_part_phi_as}
\partial_{ \bar z } \varphi( z, \lambda ) = \frac{ i \sqrt{E} }{ 2 } \lambda e^{ \frac{ i \sqrt{E} }{ 2 } \left( \lambda \bar z + \frac{ z }{ \lambda } \right) } \left( \frac{ i \sqrt{E} }{ 2 } \left( \lambda \bar z - \frac{ 1 }{ \lambda } z \right) + 1 + o( 1 ) \right), \text{ as } z \to \infty.
\end{equation}
In addition, note that
\begin{equation}
\label{app_int_vanish}
\int\limits_{ C_R } \frac{ \zeta d \bar \zeta }{ \zeta^2 } = 0, \quad \int\limits_{ C_R } \frac{ \bar \zeta d \bar \zeta }{ \zeta^2 } = 0.
\end{equation}
Thus from Lemma \ref{w_lemma}, (\ref{app_part_phi_as}), (\ref{app_int_vanish}) it follows that as $ R \to \infty $ the first summand in the right-hand side of (\ref{app_i_two}) vanishes. We apply the Stokes theorem to the second summand:
\begin{multline}
\label{app_sec_sum}
- \frac{ 4 i \sqrt{E} }{ \lambda } \iint\limits_{ B_R } \partial_{ \zeta } \left( e^{ -\frac{ i \sqrt{E} }{ 2 } ( \lambda \bar \zeta + \frac{ \zeta }{ \lambda } ) } \, w \right) \, \partial_{ \bar \zeta } \varphi \, d \Re \zeta \, d \Im \zeta = \frac{ i E }{ \lambda^2 } \int\limits_{ C_R } e^{ -\frac{ i \sqrt{E} }{ 2 } ( \lambda \bar \zeta + \frac{ \zeta }{ \lambda } ) } \, w \, \varphi \, d \zeta - \\
- \frac{ 2 \sqrt{E} }{ \lambda } \int\limits_{ C_R } e^{ -\frac{ i \sqrt{E} }{ 2 } ( \lambda \bar \zeta + \frac{ \zeta }{ \lambda } ) } \,\partial_{ \zeta } w \, \varphi \, d \zeta + \frac{ 4 i \sqrt{E} }{ \lambda } \iint\limits_{ B_R } \partial_{ \zeta } \partial_{ \bar \zeta } \left( e^{ -\frac{ i \sqrt{E} }{ 2 } ( \lambda \bar \zeta + \frac{ \zeta }{ \lambda } ) } \, w \right) \, \varphi \, d \Re \zeta \, d \Im \zeta.
\end{multline}
From Lemma \ref{w_lemma} and asymptotics of $ \varphi( z, \lambda ) $ as $ z \to \infty $ it follows that the second summand in the right-hand side of (\ref{app_sec_sum}) vanishes as $ R \to \infty $. For the first summand we obtain
\begin{equation*}
\frac{ i E }{ \lambda^2 } \int\limits_{ C_R } e^{ -\frac{ i \sqrt{E} }{ 2 } ( \lambda \bar \zeta + \frac{ \zeta }{ \lambda } ) } \, w \, \varphi \, d \zeta \sim \frac{ ( \sqrt{E} )^3 }{ 2 \lambda^3 } \frac{ 3 \hat v( 0 ) }{ \pi } \int\limits_{ C_R } \frac{ d \zeta }{ \zeta } = 3 \frac{ i ( \sqrt{E} )^3 }{ \lambda^3 } \hat v( 0 ) \text{ as } R \to \infty.
\end{equation*}

Finally, for $ I_{ 11 } + I_{ 7 } $ we obtain that
\begin{multline*}
\lim\limits_{ R \to \infty } ( I_{ 11 } + I_{ 7 } ) = - 2 \iint\limits_{ \mathbb{C} } e^{ -\frac{ i \sqrt{E} }{ 2 } ( \lambda \bar \zeta + \frac{ \zeta }{ \lambda } ) } \, \partial_{ \zeta } v \, w \, \varphi \, d \Re \zeta \, d \Im \zeta - \frac{ 6 E }{ \lambda^2 } \iint\limits_{ \mathbb{C} } e^{ -\frac{ i \sqrt{E} }{ 2 } ( \lambda \bar \zeta + \frac{ \zeta }{ \lambda } ) } \, \partial_{ \zeta } v \, \varphi \, d \Re \zeta \, d \Im \zeta - \\
- \frac{ 12 i \sqrt{E} }{ \lambda } \iint\limits_{ \mathbb{C} } e^{ -\frac{ i \sqrt{E} }{ 2 } ( \lambda \bar \zeta + \frac{ \zeta }{ \lambda } ) } \, \partial_{ \zeta }^2 \, v \, \varphi \, d \Re \zeta \, d \Im \zeta - 2 \iint\limits_{ \mathbb{C} } e^{ -\frac{ i \sqrt{E} }{ 2 } ( \lambda \bar \zeta + \frac{ \zeta }{ \lambda } ) } \, v \, \partial_{ \zeta } w \, \varphi \, d \Re \zeta \, d \Im \zeta + \\
+ 3 \frac{ i ( \sqrt{E} )^3 }{ \lambda^3 } \hat v( 0 ).
\end{multline*}

Thus it can be obtained that
\begin{multline}
\label{second_sum_alpha}
\lim\limits_{ R \to \infty }( I_{ 3 } + I_{ 4 } + I_{ 5 } + I_{ 6 } + I_{ 7 } + I_{ 8 } + I_{ 11 } + I_{ 12 } ) = \\
= -\frac{ 6 E }{ \lambda^2 } \iint\limits_{ \mathbb{C} } e^{ -\frac{ i \sqrt{E} }{ 2 } ( \lambda \bar \zeta + \frac{ \zeta }{ \lambda } ) } \, \partial_{ \zeta } v \, \varphi \, d \Re \zeta \, d \Im \zeta - \frac{ 12 i \sqrt{E} }{ \lambda } \iint\limits_{ \mathbb{C} } e^{ -\frac{ i \sqrt{E} }{ 2 } ( \lambda \bar \zeta + \frac{ \zeta }{ \lambda } ) } \, \partial_{ \zeta }^2 v \, \varphi \, d \Re \zeta \, d \Im \zeta - \\
- 6 E \lambda^2 \iint\limits_{ \mathbb{C} } e^{ -\frac{ i \sqrt{E} }{ 2 } ( \lambda \bar \zeta + \frac{ \zeta }{ \lambda } ) } \, \partial_{ \bar \zeta  } v \, \varphi \, d \Re \zeta \, d \Im \zeta - 12 i \sqrt{E} \lambda \iint\limits_{ \mathbb{C} } e^{ -\frac{ i \sqrt{E} }{ 2 } ( \lambda \bar \zeta + \frac{ \zeta }{ \lambda } ) } \, \partial_{ \bar \zeta }^2 v \, \varphi \, d \Re \zeta \, d \Im \zeta -\\
- 3 i ( \sqrt{E} )^3 \left( \lambda^3 - \frac{ 1 }{ \lambda^3 } \right) \hat v( 0 ).
\end{multline}
Finally,
\begin{equation}
\label{fourteen_alpha}
I_{ 14 } = 3 i ( \sqrt{E} )^3 \left( \lambda^3 - \frac{ 1 }{ \lambda^3 } \right) a( \lambda, t )
\end{equation}
and thus from (\ref{sum_fourteen_alpha})-(\ref{fourteen_alpha}) we obtain formula (\ref{partial_concise_alpha}).

\bigskip
\bigskip
\noindent
\emph{Derivation of (\ref{partial_concise_beta}).} Similarly, the formula (\ref{partial_concise_beta}) can be derived. Substituting (\ref{NV}) and (\ref{partial_phi}) into (\ref{partial_beta}) yields
{\allowdisplaybreaks
\begin{multline}
\label{sum_fourteen_beta}
\partial_{ t } \beta( \lambda, t ) = \lim\limits_{ R \to \infty } \Biggl( 8 \iint\limits_{ B_R } e^{ \frac{ \sqrt{ -E } }{ 2 } ( \bar \lambda \zeta + \frac{ \bar \zeta }{ \bar \lambda } ) } \, \partial_{ \zeta }^3 v \, \varphi \, d \Re \zeta \, d \Im \zeta + 8 \iint\limits_{ B_R } e^{ \frac{ \sqrt{ -E } }{ 2 } ( \bar \lambda \zeta + \frac{ \bar \zeta }{ \bar \lambda } ) } \, \partial_{ \bar \zeta }^3 v \, \varphi \, d \Re \zeta \, d \Im \zeta + \\
+ 2 \iint\limits_{ B_R } e^{ \frac{ \sqrt{ -E } }{ 2 } ( \bar \lambda \zeta + \frac{ \bar \zeta }{ \bar \lambda } ) } \, \partial_{ \zeta } v \, w \, \varphi \, d \Re \zeta \, d \Im \zeta + 2 \iint\limits_{ B_R } e^{ \frac{ \sqrt{ -E } }{ 2 } ( \bar \lambda \zeta + \frac{ \bar \zeta }{ \bar \lambda } ) } \, v \, \partial_{ \zeta } w \, \varphi \, d \Re \zeta \, d \Im \zeta + \\
+ 2 \iint\limits_{ B_R } e^{ \frac{ \sqrt{ -E } }{ 2 } ( \bar \lambda \zeta + \frac{ \bar \zeta }{ \bar \lambda } ) } \, \partial_{ \bar \zeta } v \, \bar w \, \varphi \, d \Re \zeta \, d \Im \zeta + 2 \iint\limits_{ B_R } e^{ \frac{ \sqrt{ -E } }{ 2 } ( \bar \lambda \zeta + \frac{ \bar \zeta }{ \bar \lambda } ) } \, v \, \partial_{ \bar \zeta } \bar w \, \varphi \, d \Re \zeta \, d \Im \zeta - \\
- 2 E \iint\limits_{ B_R } e^{ \frac{ \sqrt{ -E } }{ 2 } ( \bar \lambda \zeta + \frac{ \bar \zeta }{ \bar \lambda } ) } \, \partial_{ \zeta } w \, \varphi \, d \Re \zeta \, d \Im \zeta - 2 E \iint\limits_{ B_R } e^{ \frac{ \sqrt{ -E } }{ 2 } ( \bar \lambda \zeta + \frac{ \bar \zeta }{ \bar \lambda } ) } \, \partial_{ \bar \zeta } \bar w \, \varphi \, d \Re \zeta \, d \Im \zeta + \\
+ 8 \iint\limits_{ B_R } e^{ \frac{ \sqrt{ -E } }{ 2 } ( \bar \lambda \zeta + \frac{ \bar \zeta }{ \bar \lambda } ) } \, v \, \partial_{ \zeta }^{ 3 } \varphi \, d \Re \zeta \, d \Im \zeta + 8 \iint\limits_{ B_R } e^{ \frac{ \sqrt{ -E } }{ 2 } ( \bar \lambda \zeta + \frac{ \bar \zeta }{ \bar \lambda } ) } \, v \, \partial_{ \bar \zeta }^{ 3 } \varphi \, d \Re \zeta \, d \Im \zeta + \\
+ 2 \iint\limits_{ B_R } e^{ \frac{ \sqrt{ -E } }{ 2 } ( \bar \lambda \zeta + \frac{ \bar \zeta }{ \bar \lambda } ) } \, v \, w \, \partial_{ \zeta } \varphi \, d \Re \zeta \, d \Im \zeta + 2 \iint\limits_{ B_R } e^{ \frac{ \sqrt{ -E } }{ 2 } ( \bar \lambda \zeta + \frac{ \bar \zeta }{ \bar \lambda } ) } \, v \, \bar w \, \partial_{ \bar \zeta } \varphi \, d \Re \zeta \, d \Im \zeta + \\
+ i ( \sqrt{E} )^3 \left( \lambda^3 + \frac{ 1 }{ \lambda^3 } \right) \iint\limits_{ B_R } e^{ \frac{ \sqrt{ -E } }{ 2 } ( \bar \lambda \zeta + \frac{ \bar \zeta }{ \bar \lambda } ) } \, v \, \varphi \, d \Re \zeta \, d \Im \zeta + \\
+ 3 i ( \sqrt{E} )^3 \left( \lambda^3 - \frac{ 1 }{ \lambda^3 } \right) \iint\limits_{ B_R } e^{ \frac{ \sqrt{ -E } }{ 2 } ( \bar \lambda \zeta + \frac{ \bar \zeta }{ \bar \lambda } ) } \, v \, \psi \, d \Re \zeta \, d \Im \zeta \Biggr) = \lim\limits_{ R \to \infty } \sum\limits_{ i = 1 }^{ 14 } J_i.
\end{multline} }
Integrating $ J_9 $ by parts yields
\begin{multline*}
\lim\limits_{ R \to \infty } J_9 = - \bar \lambda^3 ( \sqrt{ -E } )^3 \iint\limits_{ \mathbb{C} } e^{ \frac{ \sqrt{ -E } }{ 2 } ( \bar \lambda \zeta + \frac{ \bar \zeta }{ \bar \lambda } ) } \, v \, \varphi \, d \Re \zeta \, d \Im \zeta + 6 \bar \lambda^2 E \iint\limits_{ \mathbb{C} } e^{ \frac{ \sqrt{ -E } }{ 2 } ( \bar \lambda \zeta + \frac{ \bar \zeta }{ \bar \lambda } ) } \, \partial_{ \zeta } v \, \varphi \, d \Re \zeta \, d \Im \zeta - \\
- 12 \bar \lambda \sqrt{ -E } \iint\limits_{ \mathbb{C} } e^{ \frac{ \sqrt{ -E } }{ 2 } ( \bar \lambda \zeta + \frac{ \bar \zeta }{ \bar \lambda } ) } \, \partial_{ \zeta }^2 v \, \varphi \, d \Re \zeta \, d \Im \zeta - 8 \iint\limits_{ \mathbb{C} } e^{ \frac{ \sqrt{ -E } }{ 2 } ( \bar \lambda \zeta + \frac{ \bar \zeta }{ \bar \lambda } ) } \, \partial_{ \zeta }^3 v \, \varphi \, d \Re \zeta \, d \Im \zeta.
\end{multline*}
In this way it can be obtained that
\begin{multline}
\label{first_sum_beta}
\lim\limits_{ R \to \infty } ( J_1 + J_2 + J_9 + J_{ 10 } + J_{ 13 } ) = i ( \sqrt{E} )^3 \left( \lambda^3 + \frac{ 1 }{ \lambda^3 } + ( \sgn E ) \left( \bar \lambda^3 + \frac{ 1 }{ \bar \lambda^3 } \right) \right) \beta( \lambda, t ) + \\
+ 6 E \bar \lambda^2 \iint\limits_{ \mathbb{C} } e^{ \frac{ \sqrt{ -E } }{ 2 } ( \bar \lambda \zeta + \frac{ \bar \zeta }{ \bar \lambda } ) } \, \partial_{ \zeta } v \, \varphi \, d \Re \zeta \, d \Im \zeta - 12 \sqrt{ -E } \bar \lambda \iint\limits_{ \mathbb{C} } e^{ \frac{ \sqrt{ -E } }{ 2 } ( \bar \lambda \zeta + \frac{ \bar \zeta }{ \bar \lambda } ) } \, \partial_{ \zeta }^2 v \, \varphi \, d \Re \zeta \, d \Im \zeta + \\
+ \frac{ 6 E }{ \bar \lambda^2 } \iint\limits_{ \mathbb{C} } e^{ \frac{ \sqrt{ -E } }{ 2 } ( \bar \lambda \zeta + \frac{ \bar \zeta }{ \bar \lambda } ) } \, \partial_{ \bar \zeta } v \, \varphi \, d \Re \zeta \, d \Im \zeta - \frac{ 12 \sqrt{-E} }{ \bar \lambda } \iint\limits_{ \mathbb{C} } e^{ \frac{ \sqrt{ -E } }{ 2 } ( \bar \lambda \zeta + \frac{ \bar \zeta }{ \bar \lambda } ) } \, \partial_{ \bar \zeta }^2 v \, \varphi \, d \Re \zeta \, d \Im \zeta.
\end{multline}

Integrating $ J_{ 11 } $ by parts we obtain
\begin{multline*}
\lim\limits_{ R \to \infty } ( J_{ 11 } + J_{ 7 } ) = - 2 \iint\limits_{ \mathbb{C} } e^{ \frac{ \sqrt{-E} }{ 2 } ( \bar \lambda \zeta + \frac{ \bar \zeta }{ \bar \lambda } ) } \, \partial_{ \zeta } v \, w \, \varphi \, d \Re \zeta \, d \Im \zeta + \\
+ \lim\limits_{ R \to \infty } \Biggl(  - \sqrt{- E } \bar \lambda \iint\limits_{ B_R } e^{ \frac{ \sqrt{-E} }{ 2 } ( \bar \lambda \zeta + \frac{ \bar \zeta }{ \bar \lambda } ) } \, v \, w \, \varphi \, d \Re \zeta \, d \Im \zeta - \\
- 2 E \iint\limits_{ B_R } e^{ \frac{ \sqrt{-E} }{ 2 } ( \bar \lambda \zeta + \frac{ \bar \zeta }{ \bar \lambda } ) } \, \partial_{ \zeta } w \, \varphi \, d \Re \zeta \, d \Im \zeta \Biggr) - 2 \iint\limits_{ \mathbb{C} } e^{ \frac{ \sqrt{-E} }{ 2 } ( \bar \lambda \zeta + \frac{ \bar \zeta }{ \bar \lambda } ) } \, v \, \partial_{ \zeta } w \, \varphi \, d \Re \zeta \, d \Im \zeta.
\end{multline*}

Consider
\begin{equation*}
\tilde J = - \sqrt{- E } \bar \lambda \iint\limits_{ B_R } e^{ \frac{ \sqrt{-E} }{ 2 } ( \bar \lambda \zeta + \frac{ \bar \zeta }{ \bar \lambda } ) } \, v \, w \, \varphi \, d \Re \zeta \, d \Im \zeta.
\end{equation*}
Taking into account that $ - 4 \partial_{ \zeta } \partial_{ \bar \zeta } \varphi + v \varphi = E \varphi $ we obtain
that
\begin{multline*}
\tilde J = ( \sqrt{-E} )^3 \bar \lambda \iint\limits_{ B_R } e^{ \frac{ \sqrt{-E} }{ 2 } ( \bar \lambda \zeta + \frac{ \bar \zeta }{ \bar \lambda } ) } \, w \, \varphi \, d \Re \zeta \, d \Im \zeta - \\
- 4 \sqrt{- E } \bar \lambda \iint\limits_{ B_R } e^{ \frac{ \sqrt{-E} }{ 2 } ( \bar \lambda \zeta + \frac{ \bar \zeta }{ \bar \lambda } ) } \, w \, \partial_{ \zeta } \partial_{ \bar \zeta } \varphi \, d \Re \zeta \, d \Im \zeta = \tilde J_{ (1) } + \tilde J_{ (2) }.
\end{multline*}

Applying the Stokes theorem to $ \tilde J_{ (2) } $ we get that
\begin{multline*}
\tilde J_{ (2) } = - 2 i \sqrt{- E } \bar \lambda \int\limits_{ C_R } e^{ \frac{ \sqrt{-E} }{ 2 } ( \bar \lambda \zeta + \frac{ \bar \zeta }{ \bar \lambda } ) } \, w \, \partial_{ \bar \zeta } \varphi \, d \bar \zeta  + \\
+ 4 \sqrt{- E } \bar \lambda \iint\limits_{ B_R } \partial_{ \zeta } \left( e^{ \frac{ \sqrt{-E} }{ 2 } ( \bar \lambda \zeta + \frac{ \bar \zeta }{ \bar \lambda } ) } \, w \right) \, \partial_{ \bar \zeta } \varphi \, d \Re \zeta \, d \Im \zeta = \tilde J_{ (21) } + \tilde J_{ (22) },
\end{multline*}
where $ C_R $ is the boundary of $ B_R $.

From asymptotics (\ref{app_part_phi_as}) and Lemma \ref{w_lemma} it follows that $ J_{ ( 21 ) } $ is an integral of the stationary phase type and thus it vanishes as $ R \to \infty $.

Now let us apply the Stokes theorem to $ \tilde J_{ (22) } $:
\begin{multline}
\label{app_j22}
\tilde J_{ (22) } = - 2 i \sqrt{- E } \bar \lambda \int\limits_{ C_R } e^{ \frac{ \sqrt{-E} }{ 2 } ( \bar \lambda \zeta + \frac{ \bar \zeta }{ \bar \lambda } ) } \, \partial_{ \zeta } w \, \varphi \, d \zeta + \\
+ i E \bar \lambda^2 \int\limits_{ C_R } e^{ \frac{ \sqrt{-E} }{ 2 } ( \bar \lambda \zeta + \frac{ \bar \zeta }{ \bar \lambda } ) } \, w \, \varphi \, d \zeta - 4 \sqrt{- E } \bar \lambda \iint\limits_{ B_R } \partial_{ \zeta } \partial_{ \bar \zeta } \left( e^{ \frac{ \sqrt{-E} }{ 2 } ( \bar \lambda \zeta + \frac{ \bar \zeta }{ \bar \lambda } ) } \, w \right) \, \varphi \, d \Re \zeta \, d \Im \zeta.
\end{multline}

Treating the first two summands in the right-hand side of (\ref{app_j22}) in the similar way as we treated $ \tilde J_{ (21) } $, we obtain that
\begin{equation*}
\tilde J_{ ( 2 ) } \sim \tilde J_{ (22) } \sim - 4\sqrt{- E } \bar \lambda \iint\limits_{ B_R } \partial_{ \zeta } \partial_{ \bar \zeta } \left( e^{ \frac{ \sqrt{-E} }{ 2 } ( \bar \lambda \zeta + \frac{ \bar \zeta }{ \bar \lambda } ) } \, w \right) \, \varphi \, d \Re \zeta \, d \Im \zeta, \text{ as } R \to \infty.
\end{equation*}

Finally, for $ J_{ 11 } + J_{ 7 } $ we obtain that
\begin{multline*}
\lim\limits_{ R \to \infty } ( J_{ 11 } + J_{ 7 } ) = - 2 \iint\limits_{ \mathbb{C} } e^{ \frac{ \sqrt{-E} }{ 2 } ( \bar \lambda \zeta + \frac{ \bar \zeta }{ \bar \lambda } ) } \, \partial_{ \zeta } v \, w \, \varphi \, d \Re \zeta \, d \Im \zeta - \\
- 6 E \bar \lambda^2 \iint\limits_{ \mathbb{C} } e^{ \frac{ \sqrt{-E} }{ 2 } ( \bar \lambda \zeta + \frac{ \bar \zeta }{ \bar \lambda } ) } \, \partial_{ \zeta } v \, \varphi \, d \Re \zeta \, d \Im \zeta + 12 \sqrt{- E } \bar \lambda \iint\limits_{ \mathbb{C} } e^{ \frac{ \sqrt{-E} }{ 2 } ( \bar \lambda \zeta + \frac{ \bar \zeta }{ \bar \lambda } ) } \, \partial_{ \zeta }^2 \, v \, \varphi \, d \Re \zeta \, d \Im \zeta - \\
- 2 \iint\limits_{ \mathbb{C} } e^{ \frac{ \sqrt{-E} }{ 2 } ( \bar \lambda \zeta + \frac{ \bar \zeta }{ \bar \lambda } ) } \, v \, \partial_{ \zeta } w \, \varphi \, d \Re \zeta \, d \Im \zeta.
\end{multline*}

Thus it can be obtained that
\begin{multline}
\label{second_sum_beta}
\lim\limits_{ R \to \infty } ( J_{ 3 } + J_{ 4 } + J_{ 5 } + J_{ 6 } + J_{ 7 } + J_{ 8 } + J_{ 11 } + J_{ 12 } ) = \\
= - 6 E \bar \lambda^2 \iint\limits_{ \mathbb{C} } e^{ \frac{ \sqrt{ -E } }{ 2 } ( \bar \lambda \zeta + \frac{ \bar \zeta }{ \bar \lambda } ) } \, \partial_{ \zeta } v \, \varphi \, d \Re \zeta \, d \Im \zeta + 12 \sqrt{-E} \bar \lambda \iint\limits_{ \mathbb{C} } e^{ \frac{ \sqrt{ -E } }{ 2 } ( \bar \lambda \zeta + \frac{ \bar \zeta }{ \bar \lambda } ) } \, \partial_{ \zeta }^2 v \, \varphi \, d \Re \zeta \, d \Im \zeta - \\
 -\frac{ 6 E }{ \bar \lambda^2 } \iint\limits_{ \mathbb{C} } e^{ \frac{ \sqrt{ -E } }{ 2 } ( \bar \lambda \zeta + \frac{ \bar \zeta }{ \bar \lambda } ) } \, \partial_{ \bar \zeta  } v \, \varphi \, d \Re \zeta \, d \Im \zeta + \frac{ 12 \sqrt{-E} }{ \bar \lambda } \iint\limits_{ \mathbb{C} } e^{ \frac{ \sqrt{ -E } }{ 2 } ( \bar \lambda \zeta + \frac{ \bar \zeta }{ \bar \lambda } ) } \, \partial_{ \bar \zeta }^2 v \, \varphi \, d \Re \zeta \, d \Im \zeta.
\end{multline}
Finally,
\begin{equation}
\label{fourteen_beta}
J_{ 14 } = 3 i ( \sqrt{E} )^3 \left( \lambda^3 - \frac{ 1 }{ \lambda^3 } \right) b( \lambda, t )
\end{equation}
and thus from (\ref{sum_fourteen_beta})-(\ref{fourteen_beta}) we obtain formula (\ref{partial_concise_beta}).

\end{document}